\documentclass{amsart}

\usepackage{amssymb}
\usepackage{latexsym}
\usepackage{amsmath}
\usepackage{enumerate}
\usepackage{amsthm}
\usepackage{stmaryrd}
\usepackage{pifont}
\usepackage{tikz-cd}

\newtheorem{theorem}{Theorem}[section]
\newtheorem{define}[theorem]{Definition}

\newtheorem{exa}[theorem]{Example}
\newenvironment{example}{\begin{exa} \rm}{\qee\end{exa}}
\newtheorem{exerc}[theorem]{Exercise}

\newtheorem{conj}[theorem]{Conjecture}

\newtheorem{ques}[theorem]{Open Question}
\newenvironment{question}{\begin{ques} \rm}{\qee\end{ques}}
\newtheorem{lem}[theorem]{Lemma}
\newenvironment{lemma}{\begin{lem} \it}{\end{lem}}
\newtheorem{cor}[theorem]{Corollary}
\newenvironment{corollary}{\begin{cor} \it}{\end{cor}}
\newtheorem{rem}[theorem]{Remark}
\newenvironment{remark}{\begin{rem} \rm}{\qee\end{rem}}

\DeclareMathOperator{\possible}{\text{\tikz[scale=.6ex/1cm,baseline=-.6ex,rotate=45,line width=.1ex]{
                            \draw (-1,-1) rectangle (1,1);}}}
\DeclareMathOperator{\necessary}{\text{\tikz[scale=.6ex/1cm,baseline=-.6ex,line width=.1ex]{
                            \draw (-1,-1) rectangle (1,1);}}}
\DeclareMathOperator{\gpossible}{\text{\tikz[scale=.6ex/1cm,baseline=-.6ex,rotate=45,line width=.1ex]{
                            \draw[gray, fill = gray, fill opacity = .90] (-1,-1) rectangle (1,1);}}}
 \DeclareMathOperator{\gnecessary}{\text{\tikz[scale=.6ex/1cm,baseline=-.6ex,line width=.1ex]{
                            \draw[gray, fill = gray, fill opacity = .90] (-1,-1) rectangle (1,1);}}}
                            
                                                        \DeclareMathOperator{\xpossible}{\text{\tikz[scale=.6ex/1cm,baseline=-.6ex,rotate=45,line width=.1ex]{
                            \draw (-1,-1) rectangle (1,1); \draw[very thin] (-.6,-.6) rectangle (.6,.6);}}}
\DeclareMathOperator{\xnecessary}{\text{\tikz[scale=.6ex/1cm,baseline=-.6ex,line width=.1ex]{
                            \draw (-1,-1) rectangle (1,1); \draw[very thin] (-.6,-.6) rectangle (.6,.6);}}}

\newcommand{\qee} {\hspace*{2mm}\hfill \ding{109}}

\renewcommand{\iff}{\leftrightarrow}
\renewcommand{\phi}{\varphi}

\newcommand{\qedright}{\belowdisplayskip=-12pt}

\newcommand{\sbra}[1]{\textup(#1\textup)}
\newcommand{\frc}[1]{\aco_{#1}\top}
\newcommand{\frcr}[2]{\aco_{#1}^{#2}\top}
\newcommand{\kraj}{Kraj\'{\i}{\v{c}}ek}
\newcommand{\efun}[1]{{\sf proj}_{#1}}
\newcommand{\bbot}{\bot\!\!\!\!\bot}
\newcommand{\To}{\Rightarrow}
\newcommand{\Iff}{\Leftrightarrow}
\newcommand{\stackarrow}[1]{\stackrel{#1}{\longrightarrow}}
\newcommand{\opr}{\necessary}
\newcommand{\oco}{\possible}
\newcommand{\apr}{{\vartriangle}}
\newcommand{\aco}{{\triangledown}}
\newcommand{\graydi}{\gpossible}
\newcommand{\graysq}{\gnecessary}
 \newcommand{\nrhd}{\mathrel{\not\! \rhd}}
\newcommand{\mutint}{\bowtie}
\newcommand{\verz}[1]{\{ #1 \}}
 \newcommand{\tupel}[1]{{\langle #1 \rangle}}
 \newcommand{\gn}[1]{{\underline{\ulcorner #1 \urcorner}}}
 \newcommand{\ccz}{\ensuremath{\xpossible}}
 \newcommand{\ccu}[1]{\xpossible #1}
  \newcommand{\ccd}[2]{\xpossible_{#1}#2}
  \newcommand{\ccdu}[2]{\xpossible^{#1}#2} 
  \newcommand{\cct}[3]{\xpossible^{#1}_{#2}#3}
  \newcommand{\ppz}{\ensuremath{\xnecessary}}
   \newcommand{\ppd}[2]{\xnecessary_{#1}#2}
    \newcommand{\ppt}[3]{\xnecessary^{#1}_{#2}#3}
    \newcommand{\dom}{{\sf D}}
\newcommand{\li}{\ell}
\newcommand{\paoo}{{\sf PA}^{{\downarrow}{\downarrow}{\downarrow}}}
\newcommand{\paco}{{\sf PA}^{{\downarrow}{\downarrow}}}
\newcommand{\paba}{{\sf PA}^{\downarrow}}
\newcommand{\desca}{{\sf DA}}

\title[Friedman-reflexivity]{Friedman-reflexivity\\
{\footnotesize Interpreters as Consistoids}}
\author{Albert Visser}
 \address{Philosophy, Faculty of Humanities,
                Utrecht University,
               Janskerkhof 13,
                3512BL~~Utrecht, The Netherlands}
\email{a.visser@uu.nl}
\date{\today}

\begin{document}


\thanks{I thank Harvey Friedman for his permission to use his ideas as a basis for this paper.
I am grateful to Jaap van Oosten for his advice. I thank Lev Beklemishev, Joost Joosten, Tadeusz Litak, and Fedor Pakhomov for their
questions and observations. I am greatful to Joel Hamkins for his nice macros for the modal operators.}

\begin{abstract}
In the present paper, we explore an idea of Harvey Friedman to obtain a coordinate-free presentation
of consistency. Friedman shows that, over Peano Arithmetic, the consistency statement for a finitely axiomatised
theory $A$ can be characterised as the weakest statement $C$
over Peano Arithmetic such that ${\sf PA}+C$ interprets $A$.

We study the question which base theories $U$ have the property that, for any finitely
axiomatised $A$, there is  a weakest $C$ such that $U+C$ interprets $A$. We call such theories
Friedman-reflexive. We explore various implications of Friedman-reflexiveness.

We show that a very weak theory, Peano Corto, is Friedman-reflexive. 
We do not get the usual consistency statements here, but bounded, cut-free or Herbrand consistency statements.
We illustrate that Peano Corto as a base theory has additional desirable properties.

We prove a characterisation theorem for Friedman-reflexive sequential theories. We provide
an example of a Friedman-reflexive sequential theory that substantially differs
from the paradigm cases of Peano Arithmetic and Peano Corto.

The consistency-like statements provided by a Friedman-reflexive base $U$ can be used
to define a provability-like notion for a finitely axiomatised $A$ that interprets $U$ via an interpretation
$K$ of $U$ in $A$.
We explore what modal logics this idea gives rise to. We call such logics \emph{interpreter logics}. We show that, generally, these logics
satisfy the L\"ob Conditions, aka {\sf K}4. We provide conditions for when these logics extend {\sf S}4, {\sf K}45,  and L\"ob's Logic. We show that, 
if either $U$ or $A$ is sequential, then the condition for
 extending L\"ob's Logic is fulfilled. Moreover, if our base theory $U$ is sequential and if, in addition, its interpreters can be effectively found,
 we prove Solovay's Theorem. This holds even if the provability-like operator is not necessarily representable by a predicate of
 G\"odel numbers.
 
 At the end of the paper, we briefly how successful the coordinate-free approach is.
 \end{abstract}

\maketitle

\section{Introduction} 
 `Consistoids' are analogues of consistency statements that are characterised in a `coordinate-free' way in the sense
that their characterisation does not depend on arithmetisation.  There are two reasons why such analogues are
interesting. First, if we consider standard applications of the Second Incompleteness Theorem, aka G2, we would like to
eliminate arithmetisation as a hidden parameter in the statement. The arithmetisation-free version is closer to
an `honest' mathematical theorem.
Secondly, on a more adventurous note,
consistoids are able to live also in contexts where no (full) arithmetisation is possible.
In other words, by considering coordinate-free versions, we can take consistency statements out of their
comfort zone.\footnote{A different example of a study of a G2-analogue that works in a wider context can be found in \cite{pakh:fini21}.}
Not only, do we think this is interesting in itself, but  the wider view allows us also to look at the standard cases from a more general 
vantage point.

We develop a strategy proposed by Harvey Friedman. See \cite{frie:aspe21}. The basic idea is that a consistency statement for 
a finitely axiomatised $A$
over a base theory {\sf B} is the weakest statement $C$ such that ${\sf B}+C$ interprets $A$. 
We will call such weakest statements: \emph{interpreters}.
We explain interpreters in more detail  in Subsection~\ref{powersmurf}.

We study Friedman's idea in a general setting. 
The comfort zone for his idea is the class of sequential theories.
 In this familiar context, consistoids are bounded (or cut-free or Herbrand) consistency statements 
 that appear in \emph{varying} interpretations of 
 the natural numbers in the base theory. We show that essentially number-system-hopping consistoides
 may really occur.
 
 Given a Friedman-reflexive base theory {\sf B} and an interpretation $K$ of {\sf B} in a finitely axiomatised theory $A$,
 we can define the \emph{interpreter logic} of $A$ (w.r.t. $K$). This is an analogue of the provability logic of $A$.
 We will have a first look at what principles of interpreter logic we get. We show that, generally, we have at least {\sf K}4,
 i.e., the L\"ob Conditions. However, this need not yield L\"ob's Logic, {\sf GL}, since we do not necessarily have a
 Fixed Point Lemma.\footnote{We may lack the resources to formulate and prove a G\"odel Fixed Point Lemma and, even if we
 have those, the provability-like notion need not be representable by a definable predicate.}
  We provide general conditions for obtaining {\sf S}4, {\sf K}45, and {\sf GL}. In case either
 {\sf B} or $A$ is sequential, the sufficient condition for obtaining 
L\"ob's Logic turns out to be fulfilled. In case {\sf B} sequential and the association of an interpreter to $A$ is
effective for $U$, we show that the proof of Solovay's Theorem can be given with few modifications. 
This is possible even if we do not have a definition of the provability-like modality corresponding to interpreters
as a predicate of {\sf B}. 
 
 We discuss the ins and outs of what is achieved and what is not achieved in the present paper in Section~\ref{corema}.
 
 \begin{remark}\label{moppersmurf}
 \emph{Caveat emptor.} What we do not try to do in this paper is solve the philosophical problem of when a sentence
 expresses a consistency statement.  The dialectic rather has this form. Show me your favourite arithmetisation of
 a consistency statement and I will characterise it, modulo  provable equivalence, in a coordinate-free manner.
 Thus, a mathematically acceptable notion is provided, even it, in the motivation phase,
arithmetisation with all its arbitrary choices still plays a role.
 \end{remark}

  \subsection{On Reading this Paper}
  Appendix~\ref{notarissmurf} gives basic definitions and basic facts and some references
  to further literature.
  
  We point out, at places, how our work links to elementary ideas from category theory
  (universal arrows and the like). The reader who wishes can skip this without losing the main thread
  of the paper. 
  
  Our main application to
  sequential theories demands some familiarity with \emph{sequentiality}. See, e.g., \cite{haje:meta91} and
  \cite{viss:what13}. However,
  the main development of Friedman-reflexivity and interpreter logics only asks for
  understanding predicate logic and translations/interpretations. The reader could elect
  to read the results concerning sequentiality, but ignore the proofs, and still get
  a good feeling for the main line of argument.

\section{The Main Ideas of the Paper}

We explain in more detail the main ideas of the paper.

\subsection{On the very Idea of a Base Theory}
To set the stage for our Friedman-style treatment of the Second Incompleteness Theorem, G2, and
of a variant of provability logic, we discuss the base theory/main theory distinction.

A formulation of the no-interpretation version of G2 looks like this:
\[U \nrhd ({\sf B}+{\sf Con}(U)),\] in other words, $U$ does not interpret the base theory plus its own 
consistency.\footnote{A first form of the no-interpretation version is due to Feferman. See \cite{fefe:arit60}. Feferman's
form was stated for extensions of {\sf PA} and for the case that main theory and base coincide.}
Here  ${\sf Con}(U)$ is the consistency statement, where we allow various further specifications of what it 
could be. Also, for a concrete formulation, there may be further conditions on $U$.

A first reason to set things up in this format is simply the nice general form of the statement.
E.g., we have Pudl\'ak-style G2:
$U \nrhd ({\sf Q}+ \oco_\alpha \top)$.\footnote{Pavel Pudl\'ak contributed the main ingredient of the result
in the present formulation. See \cite{pudl:cuts85}.} Here $\oco_\alpha\top$ is
a specific form of the consistency statement for $U$, where we formalised consistency in
arithmetic is a sufficiently good way and where $\alpha$ is a $\Sigma^0_1$-formula representing
the axiom set of $U$. The statement is general since there is no conditionalisation to `theories that
are sufficiently strong' and the like. Also, it is strong since {\sf Q} is very weak, so the
contribution of the base to the non-interpretability is minimal.
We note that Pudl\'ak's proof of this version of G2 does contain G\"odel's
original argument but extends it with new ideas like Solovay's method of shortening cuts. So perhaps, we can also say that
the Pudl\'ak-style formulation counts as a strengthening of G2.

A second reason is more proof-oriented.
Consider, for example, G2 for {\sf ZF}. 
We can easily formalise G2 in {\sf ZF} using the set-theoretical resources for coding sequences and the like.
In fact, this is easier than formalisation in {\sf PA} with only zero, successor, plus and times in the 
signature.  On the other hand, {\sf ZF} has a standard interpretation of {\sf PA} in the finite von Neumann
ordinals. We can simply import the arithmetisation of syntax as usually done in {\sf PA} in {\sf ZF} via
the von Neumann interpretation. The advantage of doing it like this that we see that G\"odel's
\emph{proof} works uniformly across theories as soon as we have an interpretation of a suitable base.
In this context, the best version is $U \nrhd ({\sf S}^1_2+ \oco_\alpha \top)$. Here ${\sf S}^1_2$ is
Buss's weak arithmetical theory for the study of p-time 
computability.\footnote{See \cite{buss:boun86} or \cite{haje:meta91} for the basic development of ${\sf S}^1_2$.}
Without any additional effort, as compared to, e.g., {\sf PA}, G\"odel's reasoning can be
 repeated in ${\sf S}^1_2$. In fact, ${\sf S}^1_2$ is better, since it prevents all kinds of
 silly and inefficient choices of doing the arithmetisation. Moreover, Pudl\'ak's argument for his strengthened version
 can be framed as first showing that $({\sf Q}+\oco_\alpha \top) \mutint ({\sf S}^1_2+ \oco_\alpha\top)$
 and, then, concluding the {\sf Q}-version from the ${\sf S}^1_2$-version. 
 We think that the ${\sf S}^1_2$-version should be viewed as G2 proper, since it
 is proved simply by G\"odel's original reasoning.
 
 A third reason is mathematico-philosophical. The usual versions of G2 depend on certain design choices.
 Can we choose the base theory in such a way that either design choices become so natural
 that they are, as it were, intrinsically given, or, in such a way that they are fully
 eliminated? 
 \begin{itemize}
 \item
 The first idea  is explored by Volker Halbach and Graham Leigh in a forthcoming book, \cite{halb:road21}.
 Let us point out how natural this idea is. Arithmetisation can be viewed as the development
 of a chain of interpretations ending in an appropriate syntax theory.
 Why not take this syntax theory as base? However, many syntax theories are possible. Which is the right one?
 Also, how do we get a syntax theory that really determines which choices to make, e.g., to represent a proof?
 \item
 The second idea is explored in this paper. We develop an idea of Harvey Friedman to eliminate
 design choices entirely in our base theory. Friedman's idea does not always deliver the usual consistency
 statements but also other things: `consistoids'.  
We elaborate on the idea in Subsection~\ref{powersmurf}.

In Subsection~\ref{moresmurf}, we will explain how the idea of main/base works out
when we consider provability-like logics based on consistoids.
\end{itemize}

Finally, there is a fourth reason. There has been philosophical discussion on the question:
when does a predicate logical sentence really express a consistency statement? 
One line of argument could be that we choose as base theory a meaningful theory and that
it is the semantics of the base theory that carries the main burden of meaning giving. 
See also \cite{viss:seco16} for some further discussion.

\subsection{Interpretation Power}\label{powersmurf}
What makes a sentence an analogue of a consistency statement? We zoom in on an answer proposed by Harvey Friedman, in \cite{frie:aspe21},
 to wit, that
the hallmark of consistency statements is the kind \emph{interpretation power} typical for consistency statements in virtue of
 the Interpretation Existence Lemma (see \cite{viss:inte18} for a detailed exposition of this lemma). The lemma says, roughly, that we can interpret
a theory $U$ in a suitable base theory plus a consistency statement for $U$.  

In case {\sf B} extends ${\sf S}^1_2$, $V$ is RE and $\alpha$  is $\Sigma^0_1$, we have the following formulation of Interpretation Existence: 
$({\sf B}+\oco_\alpha \top) \rhd V$.
We will say that any {\sf B}-sentence $B$ such that $({\sf B}+B) \rhd V$ is a \emph{pro-interpreter} of $V$ (over {\sf B}).
So, Interpretation Existence tells us that $\oco_\alpha\top$ is a pro-interpreter.
A study of pro-interpreters was undertaken in \cite{viss:seco12}.

A disadvantage of the idea of pro-interpreters is that they are not uniquely determined over the base theory.
To make them unique, we have to impose an extra demand. The obvious one is that we consider the
\emph{weakest} pro-interpreter with respect to {\sf B}-provability. Let us call such a weakest pro-interpreter simply \emph{an interpreter}.

In this paper, we will restrict ourselves to the case of consistency analogues for \emph{finitely axiomatised} theories $A$.
Needless to say that this assumption simplifies a lot.
We see that an interpreter of $A$ over {\sf B} is a {\sf B}-sentence $C$ such that
$({\sf B}+C)\rhd A$ and, for all {\sf B}-sentences $B$, we have, if $({\sf B}+B) \rhd A$, then ${\sf B}+B\vdash C$. 
Alternatively, we can say that $C$ is an interpreter of $A$ over {\sf B} iff, for all {\sf B}-sentences $B$,
$({\sf B}+B) \rhd A$ iff ${\sf B}+B \vdash C$. 

Unfortunately, over the base theory ${\sf S}^1_2$, this idea will not work. Consider any consistent $A$ such that
${\sf S}^1_2 \nrhd A$. E.g., $A$ could be Elementary Arithmetic {\sf EA}. Suppose we had an interpreter $C$ of $A$,
Then, we have, for any ${\sf S}^1_2$-cut $I$,\footnote{Our cuts are downward closed w.r.t. $<$ and
are closed under zero, successor, addition, multiplication and $\omega_1$, i.e., the function $x \mapsto 2^{({}^2{\sf log}(x))^2}$.} 
that $({\sf S}^1_2+ \oco^I_A \top) \rhd A$.
So, for all $I$, we find ${\sf S}^1_2+ \oco^I_A\top \vdash C$. In other words,
for all $I$, we find ${\sf S}^1_2+ \neg\, C  \vdash \opr^I_A\bot$. 
Since ${\sf S}^1_2 \nrhd A$, we have ${\sf S}^1_2 \nvdash C$ and, hence,
${\sf S}^1_2+ \neg\, C$ is consistent. By Theorem~\ref{hulpsmurf},
we find that $\opr_A\bot$ is true, and, thus, that $A$ is inconsistent, which contradicts our 
assumption.\footnote{A stronger version of this insight is
given in  Theorem~\ref{restraintsmurf}.}
Thus, Friedman's idea forces us towards other bases.

We call a theory that does have  the desirable property that each finitely axiomatised $A$ has an interpreter over the theory
  \emph{Friedman-reflexive}.\footnote{The ratio
behind the second name will become clear in the paper.}
Friedman's example of such a theory is Peano Arithmetic {\sf PA}. He shows that, indeed, {\sf PA} is Friedman-reflexive.
See Section~\ref{parapea}.
A disadvantage of the choice of {\sf PA} is that it is very strong.
 Our proposal for the ideal Friedman-reflexive theory is Peano Corto or $\paco$. We will show, in
 Subsection~\ref{cortoboven}, that $\paco$ has some good further properties as a base. 

\section{Basics}\label{basics}
In this section we give the basic definitions and state and prove some basic facts.

\subsection{Definitions}
 The variables $T,U,V,\dots$ range over theories in finite signature. 
These theories, generally, need not be RE. We allow inconsistent theories as values.

The variables $A,B,\dots$ range ambiguously over sentences and over
 finitely axiomatised theories. We confuse the finite conjunction of the finitely axiomatised
 theory $A$ with a single sentence $A$.
 
 There is a subtle point here. Let $U$ be any theory and let $\tau$ be an interpretation 
 from the $A$-language to the $U$-language. Let $E_A$ be the finite conjunction of the
 identity axioms for the signature of $A$.\footnote{Here it is essential that we consider
 the signature of $A$ \emph{qua theory}. Not all symbols of this signature need to actually occur in the
 conjunction of the axioms of $A$.}  We assume that $\exists x\, x=x$ is among these
 axioms. Then, $U+(E_A\wedge A)^\tau$ interprets $A$ with an interpretation based on $\tau$.
 The reason that we need $E_A$ for this is that, in first order logic, identity is treated as a
 logical constant but that in the definition of translation this is ignored and identity may be
 translated to some $U$-formula. The inclusion of $\exists x\, x=x$ in the identity axioms
 is needed since we, somewhat unnaturally, assume non-empty domains.

We say that $C$ is \emph{an interpreter of $A$ over $U$} iff, for all $B$ in the language of $U$, we have: $(U+B) \rhd A$ iff
$U+B \vdash C$.

 A theory $U$ is  \emph{Friedman-reflexive} iff all finitely axiomatised $A$ have an interpreter
 $C$ over $U$.
 
 Suppose $U$ is Friedman-reflexive. 
 We will use $\ccz(\cdot)$ for a function that selects, for each $A$, an interpreter $C$ of $A$ over $U$. 
 So, \ccz\ is only uniquely determined modulo $U$-provability.
 We write $\ccd A B$ for $\ccu {(A \wedge B)}$. If we want to make the dependence on $U$ explicit, we write
 $\ccz_{(U)}$, $\ccz_{(U),A}$, etcetera. However, we prefer to treat the dependence on $U$
 as contextually given as much as possible. We write $\ppz$ for $\neg\ccz \neg$ and $\ppz_A$ for $\neg\ccz_A \neg$.
 
 It is easy to see that $\ccu A$ and $\ccd A\top$ are equivalent over $U$. 
 We will use the second in contexts where we are interested in extensions of $A$.
 
The theory $U$ is \emph{effectively Friedman-reflexive} iff we can choose \ccz\ to be recursive. 
Note the existential quantifier here: there may very well be both recursive and non-recursive choices of
\ccz. In fact, we will see a salient example of that possibility.
 
\subsection{The Categorical Viewpoint}
We can view what is going on here in category theoretic terms as follows.
Let $\mathbb B_U$ be the partial order category of all finite extensions
of the base theory $U$ (in the same language) with order $\subseteq$, i.e. extension in the same language.
Let $\mathbb D$ be the partial pre-order category of theories ordered by interpretability $\lhd$.\footnote{In the light of our specific
application, we can replace $\lhd$ by $\lhd_{\sf loc}$.}  Let $\efun U$ be the projection functor
 from  $\mathbb B_U$ into $\mathbb D$. Then, $C$ is an interpreter of $A$ iff $A \lhd (U+C)$ is a universal arrow from $A$ to $\efun U$. 

 \begin{remark}
 We note that if we restrict $\mathbb D$ to a subcategory that contains both the finitely axiomatised theories and the
 finite extensions of $U$, but that preserves the arrows between the objects, then universality is preserved in both directions.
 
 So, e.g., in case $U$ is RE, we can restrict $\mathbb D$ to the category of RE theories with the same effect.
 \end{remark}
 
  We can view Friedman-reflexivity as follows. Let $\mathbb D_{\sf fin}$ be the  partial pre-order category of
 finitely axiomatised theories ordered by interpretability. Let {\sf emb} be the embedding functor from 
 $\mathbb D_{\sf fin}$ to $\mathbb D$. Friedman-reflexivity tells us that, for each $A$, there is a universal arrow from ${\sf emb}(A)$ to
 $\efun U$. 

 \subsection{Semi Normal Form}
 Interpreters have a kind of non-unique normal forms.
 
 \begin{theorem}
 Suppose $C$ is an interpreter of $A$ over $U$. Then, there is a translation $\tau$ of the $A$-language in the
 $U$-language, such that $U \vdash C \iff (E_A \wedge A)^\tau$.
 \end{theorem}
 
 \noindent
 We note that $\tau$ need not be unique.
 
 \begin{proof}
 Suppose $C$ is an interpreter of $A$ over $U$. 
 We have $(U+C) \rhd A$. Let $\tau$ be the translation underlying a witnessing interpretation.
 Then $U+C \vdash (E_A \wedge A)^\tau$. 
 
 On the other hand,
 $(U+(E_A \wedge A)^\tau) \rhd A$. So, by the defining property of interpreters, $U+(E_A \wedge A)^\tau\vdash C$.
 \end{proof}

 \subsection{Some Basic Facts}
We show that interpreters are unique.

\begin{theorem}
Suppose both $C$ and $C'$ are interpreters of  $A$ over $U$. Then, we have $U \vdash C \iff C'$.
\end{theorem}

\begin{proof}
Since $(U+C) \vdash C$, if follows that$(U+C) \rhd A$, and, hence $(U+C') \vdash C$. The converse direction is similar.
\end{proof}
 
 \begin{remark}
 The above uniqueness argument is  a special case
 of the usual argument for uniqueness of universal arrows.
 \end{remark}

 We show that \ccz\ is has a functorial property w.r.t. a Friedman-reflexive theory.
 
 \begin{theorem}\label{minismurf}
 Let $U$ be Friedman-reflexive and suppose $A \rhd B$. Then, we have $U \vdash \ccu A \to \ccu B$.
 \end{theorem}
 
 \begin{proof}
 Let $U$ be Friedman-reflexive. Suppose $A\rhd B$. Then, $(U+\ccu A) \rhd A \rhd B$.
 So,   $U \vdash \ccu A \to \ccu B$.
 \end{proof}
 
 \begin{remark}\label{cateen}
 Theorem~\ref{minismurf} is a special case of an elementary category theoretic insight.
 Suppose we have categories $\mathcal A$, $\mathcal B$, $\mathcal C$ and
 functors $F:\mathcal A \to \mathcal C$, $G: \mathcal B \to \mathcal C$.
 Suppose further that, for every $a$ in $A$ there is a universal arrow from
 $F(a)$ to $G$. Then, there is a functor $H: \mathcal A \to \mathcal B$, such that
 our promised universal arrows are a natural transformation from $F$ to $G\circ H$.
 \end{remark}
 
Here is a fact that does not follow from the general categorical ideas but is specific to our setting.
 
 \begin{theorem}\label{entwederodersmurf}
  Let $U$ be Friedman-reflexive. Then \ccz\ commutes, modulo $U$-provable equivalence, with finite
  disjunctions of sentences in the same signature, including the empty one. In other words, $U \vdash \neg\, \ccu\bot$ and,
  if $A$ and $B$ have the same signature, then $U \vdash \ccu{(A\vee B)} \iff (\ccu A \vee  \ccu B)$.
 \end{theorem}
 
 \begin{proof}
 We have $(U+ \ccu{\bot}) \rhd \bot$. So, $U \vdash \neg \ccu{\bot}$.
 
 We have $A \vdash (A\vee B)$ and, hence, $A \rhd (A\vee B)$. It follows, by Theorem~\ref{minismurf},
 that $U \vdash \ccu A \to  \ccu{(A\vee B)}$. Similarly, $U \vdash  \ccu B \to  \ccu{(A\vee B)}$.
 Ergo,  $U \vdash  (\ccu A\vee  \ccu B) \to  \ccu{(A\vee B)}$.
 
We have $(U+ \ccu{(A\vee B)}) \rhd (A\vee B)$. Let $\tau$ be the underlying translation of some witnessing interpretation.
 We have:\qedright
\begin{eqnarray*}
 U+ \ccu{(A\vee B)} &\vdash& (E_A\wedge (A\vee B))^\tau \\
 &\vdash & (E_A\wedge A)^\tau\vee (E_A\wedge B)^\tau \\
 & \vdash & \ccu A \vee  \ccu B
 \end{eqnarray*}
 \end{proof}
 
 We define $W:=U \owedge V$ as follows. The signature of $W$ is the disjoint union of the signatures of $U$ and $V$ 
plus two unary domain predicates $\triangle_0$ and $\triangle_1$. We have the axioms of $U$ relativised to $\triangle_0$,
the axioms of $V$ relativised to $\triangle_1$ plus axioms that say that the $\triangle_i$ form a partition of the domain.

The following fact again follows from the categorical framework alone in combination with the fact
that $\owedge$ is a supremum operator in $\mathbb D_{\sf fin}$.

\begin{theorem}
Suppose $U$ is Friedman-reflexive and $A$ and $B$ are finitely axiomatised.
Then, $U \vdash  \ccu{(A \owedge B)} \iff (\ccu A \wedge  \ccu B)$.
\end{theorem}

\begin{proof}\qedright
\begin{eqnarray*}
U+ D \vdash  \ccu{(A \owedge B)} & \Iff & (U+D)  \rhd (A\owedge B) \\
& \Iff & (U+D) \rhd A \text{ and }  (U+D) \rhd B \\
& \Iff & U+D \vdash \ccu A \text { and }  U+D \vdash  \ccu B \\
& \Iff & U+D \vdash \ccu A \wedge \ccu B
\end{eqnarray*}
\end{proof}
 
 \noindent
 We will meet $\owedge$ again in Corollary~\ref{insmurf}.
 
 \subsection{Alternative Characterisation}
 We have an alternative characterisation of Friedman-reflexivity that gives us a full adjunction.
 We write $U \subseteq V$ for $V$ is an extension of $U$ in the same language.
 
 \begin{theorem}
 A theory $U$ is Friedman-reflexive iff \textup{(\dag)} for every theory $W$, there is a theory $\mathfrak C(W) \supseteq U$,
 such that, for all $V\supseteq U$, we have $V \rhd_{\sf loc} W$ iff $V \supseteq \mathfrak C (W)$.
 \end{theorem}
 
 \begin{proof}
 Suppose $U$ is Friedman-reflexive. We prove (\dag). Consider any $W$.
 We define $\mathfrak C(W) := U + \verz{\ccu A \mid W \vdash A}$.
Let $V\supseteq U$. 
 We have:
 \begin{eqnarray*}
 V \rhd_{\sf loc} W & \Iff & \forall A\, (W \vdash A \To V \rhd A) \\
 & \Iff &  \forall A\, (W \vdash A \To \exists D\, (V\vdash D \text{ and } (U+D) \rhd A)) \\
 & \Iff & \forall A\, (W \vdash A \To \exists D\, (V\vdash D \text{ and } U+D \vdash \ccu A) )\\
  & \Iff & \forall A\, (W \vdash A \To V \vdash \ccu A) \\ 
  & \Iff & V \supseteq \mathfrak C(W)
 \end{eqnarray*}
 
 Suppose (\dag). Consider any finitely axiomatised $A$. 
 We have $\mathfrak C(A) \rhd A$. It follows that for some $C$, we have $\mathfrak C(A) \vdash C$ and $C \rhd A$.
  Since $(U+C) \rhd A$, we have $(U+C) \supseteq \mathfrak C(A)$. Thus, $C$ axiomatises $\mathfrak C(A)$ over $U$.
  It follows that:
  \begin{eqnarray*}
  (U+B) \rhd A & \Iff &   (U+B) \rhd_{\sf loc} A \\
 & \Iff & (U+B) \supseteq \mathfrak C(A)\\
 & \Iff & (U+B) \vdash C
 \end{eqnarray*}
So, we can take
$\ccu A := C$. 
\end{proof}

\noindent
 We translate our alternative characterisation in categorical terms.
 Let $\mathbb B^+_U$ be the category of all extensions of $U$ in the same language
with as arrows $\subseteq$.
 Let $\mathbb D_{\sf loc}$ be the category of all theories with the  local interpretability relation as arrows.
Let ${\sf proj}^+_U$ be the projection functor of $\mathbb B^+_U$ into $\mathbb D_{\sf loc}$. 
Then $U$ is Friedman-reflexive iff  ${\sf proj}^+_U$ has a left adjoint $\mathfrak C$.
 
 \subsection{Polyglotticity}
 A theory $U$ is \emph{polyglot} or \emph{polyglottic} if, for every consistent finitely axiomatised  $A$, there is
 a pro-interpreter $B$ of $A$ such that $U+B$ is consistent.
 
 We remind the reader that $T$ locally tolerates $V$ if, for every finite sub-theory $A$ of $V$, there is a translation
 $\tau$ such that $T$ is consistent with $(E_A \wedge A)^\tau$.

 \begin{theorem}\label{polyq}
 $U$ is polyglot iff $U$ locally tolerates the theory ${\sf Q}$ plus the true $\Pi^0_1$-sentences.
 \end{theorem}
 
 \begin{proof}
 Suppose $U$ is polyglot. Let $A$ be a finite sub-theory of  ${\sf Q}$ plus the true $\Pi^0_1$-sentences.
 Then, for some $B$, we have $U+B$ is consistent and $(U+B) \rhd A$.
 Let $\tau$ be the translation on which the interpretation of $A$ in $U+B$ is based. 
 Then, $U$ is consistent with $(E_A+A)^\tau$.
 
 Conversely, suppose $U$ locally tolerates the theory ${\sf Q}$ plus the true $\Pi^0_1$-sentences.
 Consider any finitely axiomatised consistent theory $D$. Then, $\oco_D\top$ is true.
 So, for some $\tau$, we have $U+(E_{\sf Q}\wedge{\sf Q}\wedge \oco_D\top)^\tau$ is consistent.
 By the Interpretation Existence Lemma, we find that  
 $(U+(E_{\sf Q}\wedge{\sf Q}\wedge \oco_D\top)^\tau) \rhd ({\sf Q}+\oco_D\top) \rhd D$.
 \end{proof}
 
  \noindent Since {\sf Q} interprets ${\sf S}^1_2$ on a cut, we have the same result with ${\sf S}^1_2$ substituted for {\sf Q}.

 \subsection{A Computational Insight}
 We end this section with two computational results about effectively Friedman-reflexive theories.
 
  \begin{theorem}\label{luiesmurf}
 If $U$ is consistent and effectively Friedman-reflexive, then $U$ is essentially undecidable.
 \end{theorem}
 
 \begin{proof}
 Suppose $U$ is consistent and effectively Friedman reflexive. Suppose $U$ had a consistent and decidable extension $V$.
 By Theorem~\ref{upsmurf}, the theory $V$ is effectively Friedman reflexive with recursive \ccz.
  
 Let $[S]$ be a theory of a witness of $\Sigma^0_1$-sentence $S$. See \cite{viss:onq17} for details.
 If $S$ is true, then $[S]$ has a finite model and, thus,
 any theory interprets $[S]$. It follows that $V \vdash  \ccu{[S]}$.
  On the other hand, if $[S]$ is false, then it  extends {\sf R} and, thus, since $V$ is decidable,  $V \nrhd [S]$.
  Ergo, $V \nvdash  \ccu{[S]}$.
It follows that using the decidability of $V$, we can solve the halting problem. \emph{Quod non.}
 \end{proof}
 
 We remind the reader that $T$ tolerates $U$ iff, for some translation $\tau$, the theory $T+E_U^\tau + U^\tau$ is consistent.
 In other words, $T$ tolerates $U$, if some consistent extension of $T$ interprets $U$.
 We define:
 \begin{itemize}
 \item
 A theory $U$ is \emph{strongly essentially undecidable} iff every theory $T$ that tolerates $U$ is undecidable.
 \end{itemize} 
 If a finitely axiomatised theory is essentially undecidable it is easily seen that is is also strongly essentially undecidable.
 The same does not have to hold for RE theories. See, e.g., \cite{shoe:degr58} for examples. Cobham showed that
 the Tarski-Mostowski-Robinson theory {\sf R} is strongly essentially undecidable. See \cite{vaug:theo62} or  \cite{viss:onq17} . We have:
 
 \begin{theorem} 
 Suppose $U$ is consistent, RE, and effectively Friedman-reflexive. Then, $U$ is strongly essentially undecidable.
 \end{theorem}
 
 \begin{proof}
 Let $U$ be consistent, RE, and effectively Friedman-reflexive.
 Suppose $T$ is decidable and suppose $V := T+E_U^\tau + U^\tau$ is consistent.
 In case $S$ is true, we have $V \vdash \ccdu{\tau}{[S]}$. Now suppose $[S]$ is false.
 In case $V+ \ccdu{\tau}{[S]}$ were consistent, it would follow that  $U$ tolerates $[S]$,
 contradicting the fact that $[S]$ is finitely axiomatised and essentially undecidable.
 So, $V\vdash \neg\, \ccdu{\tau}{[S]}$. Since $V$ is RE, this gives us a procedure to decide
 the halting problem. 
 \end{proof}
 
 \begin{question}\label{q1}
 Is there a theory $U$ that is consistent, effectively Friedman-reflexive and not strongly essentially undecidable?
 \end{question}
 
 \section{The Paradigm Case of Peano Arithmetic}\label{parapea}
 Peano Arithmetic is a paradigmatic theory that is Friedman-reflexive. This is the theory for which Harvey's original 
 observation was made. Our Theorem~\ref{friedmansmurf} is Theorem~2.7 of \cite{frie:aspe21}
 
\begin{theorem}[Friedman]\label{friedmansmurf}
The sentence $\oco_A\top$ is an interpreter of $A$ over {\sf PA}.
So, {\sf PA} is effectively Friedman-reflexive.
\end{theorem}
 
 \begin{proof}
 Suppose $({\sf PA}+B) \rhd A$. Then, for some finite sub-theory $D$ of {\sf PA}, we have $(D+B)\rhd A$. If follows that
 ${\sf PA} \vdash ((D+B) \rhd A)$ and, so, ${\sf PA} \vdash \oco_DB \to \oco_A\top$. By the essential reflexiveness of
 {\sf PA}, we may conclude ${\sf PA}+B \vdash \oco_A\top$.
 
 For the other direction, suppose ${\sf PA}+B \vdash \oco_A\top$. Then, by the Interpretation Existence Lemma (see \cite{viss:inte18}), 
 we find $({\sf PA}+B)\rhd A$. 
 \end{proof}
 
 We note that we have characterised the consistency statement $\oco_A\top$ among arihmetical sentences modulo {\sf PA}-provable equivalence.
We will later discuss how to improve this to {\sf EA}-provable equivalence---see Theorem~\ref{pichar}.

\section{Closure Properties}
In this section, we study various closure properties of Friedman reflexive theories.

\begin{theorem}\label{upsmurf}
Suppose $U \subseteq V$, where $V$ is in the same language as $U$, and $U$ is \sbra{effectively} Friedman-reflexive.
Then, $V$ is \sbra{effectively} Friedman-reflexive with the same interpreters. In other words, \ccz\ can be chosen the same for $U$ and $V$.
\end{theorem}

\begin{proof}
Suppose $U \subseteq V$ and $U$ is Friedman-reflexive. 
We have:
\begin{eqnarray*} 
(V+B) \rhd A & \Iff & \text{for some $D$, we have $V \vdash D$ and $(U+(D\wedge B)) \rhd A$} \\
& \Iff & \text{for some $D$, we have $V \vdash D$ and $(U+(D\wedge B)) \vdash \ccu A$}  \\
& \Iff & (V+B) \vdash \ccu A
\end{eqnarray*}
Since \ccz\ is preserved, we, \emph{ipso facto}, preserve effectivity.
\end{proof}

We note that polyglotticity is definitely \emph{not} preserved to consistent extensions.

\begin{theorem}\label{kwaaksmurf}
Suppose $U$ is  Friedman-reflexive. Suppose $A$ is consistent and $U \nrhd A$.
Then, $U+ \neg \ccu A$ is consistent and  not polyglot.
\end{theorem}

\begin{proof}
We assume the conditions of the theorem.
Since $U \nrhd A$, we have $U \nvdash \ccu A$ and, so, $U+\neg \ccu A$ is consistent.

By Theorem~\ref{upsmurf}, we may choose \ccz\ for $U+ \neg \ccu A$ the same as for $U$,
but, $\ccu A$ is inconsistent with  $U+ \neg \ccu A$.
\end{proof}
An example of our theorem is the fact that no consistent extension in the same language of ${\sf PA}+ \opr_{{\sf ACA}_0}\bot$ interprets ${\sf ACA}_0$.

Our next closure property is categorical in nature.
Let $\mathbb E := {\sf INT}^+_3$ be the category of theories in finite signature and interpretations, where  
two interpretations $K,K': T \to W$ are the same iff, for all
$T$-sentences $A$, we have $W \vdash A^K \iff A^{K'}$. We note that we do not demand that the theories are RE.
In a sense, $\mathbb E$ in combination with its sub-category of all theories with as arrows 
theory-extensions-in-the-same-language is the natural home for the study of Friedman-reflexivity.\footnote{For the case of finitely axiomatised theories,
the interaction between interpretation and extension was studied in \cite{viss:exte21}.}

\begin{theorem}\label{terugtreksmurf}
Suppose $V$ is an $\mathbb E$-retract of $U$ and
suppose  $U$ is \sbra{effectively} Friedman-reflexive.
Then, $V$ is \sbra{effectively} Friedman-reflexive. 
\end{theorem}

\begin{proof}
Let $K:V \to U$ and $M:U \to V$ witness the retraction.
So, $M\circ K$ is the same, in the sense of $\mathbb E$, as ${\sf Id}_V$.
We have:
\begin{eqnarray*} 
(V+B) \rhd A & \To &  (U+B^K) \rhd A \\
& \To & U+ B^K \vdash \ccu A \\
 & \To & V+B^{KM} \vdash \ccdu M A \\
  & \To & V+B \vdash \ccdu MA \\
  & \To & (V+B) \rhd (U+\ccu A) \\
  & \To & (V+B) \rhd A
\end{eqnarray*}

\noindent
It follows that $\ccdu MA$ is the desired interpreter of $A$ over $V$. Clearly, composition of \ccz\ with $(\cdot)^M$ preserves
effectiveness.
\end{proof}

\begin{theorem}
Suppose $U+D$ and $U+\neg\, D$ are both \textup(effectively\textup) Friedman-reflexive. Then, $U$ is also \textup(effectively\textup) Friedman-reflexive.
\end{theorem}

\begin{proof}
Consider any $A$ and let $C$ and $C'$ be the interpreters of $A$ over $U+D$, resp. $U+\neg D$. Let $C\tupel{D} C'$ be
$(D\wedge C) \vee (\neg\, D \wedge C')$.
We have:\qedright
\begin{eqnarray*}
(U+B) \rhd A & \Iff & (U+D+B) \rhd A \text{ and } (U+\neg\, D+B) \rhd A \\
& \Iff & (U+D+B) \vdash C \text{ and } (U+\neg\, D+B) \vdash C' \\
& \Iff &  (U+D+B) \vdash C\tupel{D}C' \text{ and } (U+\neg\, D+B) \vdash C \tupel{D} C' \\
& \Iff & (U+B) \vdash C\tupel{D}C'
\end{eqnarray*}
\end{proof}

In the next two theorems, we verify general insights for universal arrows in our
specific case.

Given theories $U$ and $V$, we define $W:= U \ovee V$ as follows. The signature of $W$ is the
disjoint union of the signatures of $U$ and $V$ plus a new 0-ary predicate symbol $P$.
The axioms of $W$ are $P \to A$, for axioms $A$ of $U$ and $\neg P \to B$ for axioms $B$ of $V$.
We have: 
\begin{theorem}
Suppose $U$ and $V$ are  \sbra{effectively}\ Friedman-reflexive. Then, $U\ovee V$ is \sbra{effectively} Friedman-reflexive.
\end{theorem}

\begin{proof}
Let $W:= U \ovee V$.
We have:
\begin{eqnarray*}
(W+B) \rhd A & \Iff & (U+B[P:= \top]) \rhd A\text{ and } (V+B[P:= \bot]) \rhd A \\
& \Iff & U+B[P:= \top] \vdash \ccd{(U)}A \text{ and } V+B[P:= \bot] \vdash \ccd{(V)}A\\
& \Iff & W + B \vdash (\ccd{(U)}A)\tupel{P}(\ccd{(V)}A)  
\end{eqnarray*}
So, we can take $\ccd{(W)}A :=(\ccd{(U)}A)\tupel{P}(\ccd{(V)}A)$.
\end{proof}

 \section{Interpreter Logics}\label{intlog}
 Can something like a modal logic be based on interpreters as an analog of provability logic? 
 Since we only consider interpreters for finitely axiomatised theories, this should be a modal logic
 interpreted in a finitely axiomatised theory. We first give the definitions and then some motivating remarks.
 
 \subsection{Definitions}
 An \emph{FM-frame} is a pair $\tupel{A,U}$, where $A$ is finitely axiomatised, $U$ is Friedman-reflexive and $A \rhd U$.\footnote{`FM'
 stands for `Feferman and Montague' who initiated the idea of looking at the combination of G2 and interpretability.}
 The interpretation $K:A \rhd U$, where $\tupel{A,U}$ is an FM-frame, is an FM-interpretation. 
 We consider $U$ and $A$ as part of the data for $K$.

Consider an FM-interpretation $K:U \lhd A$.
 We define $\graydi_{K,B} C := \cct{K}{(U),B}C$. 
 Our default case is where $A $ is identical to $B$. In this case we write $\graydi_{K} C$. 
Also, in many cases, we treat $K$ as contextually given and simply write $\graydi$. As usual, we set $\graysq := \neg \graydi\neg$. 

We consider the usual modal language with possibility operator $\oco$ and with necessity  defined by $\neg\oco\neg$.
Let $\sigma$ be a function from the propositional atoms to the $A$-language. We define $\phi^{(\sigma,K)}$ as follows.
\begin{itemize}
\item
$p^{(\sigma,K)} := \sigma(p)$.
\item
$(\cdot)^{(\sigma,K)}$ commutes with the truth-functional connectives.
\item
$(\oco \psi)^{(\sigma,K)} := \graydi_{K,A}\psi^{(\sigma,K)}$.
\end{itemize}
 
 We will call the logic of $\graydi_{K,A}$ over $A$: $\Lambda^{\sf fr}_K$.
 So, 
 \begin{itemize}
 \item
$\Lambda^{\sf fr}_K = \verz{\phi \mid \text{ for all $\sigma$, we have }A \vdash \phi^{(\sigma,K)}}$.
 \end{itemize}
 
 We also define the logic of a frame $\tupel{A,U}$.
   \begin{itemize}
 \item
 $\Lambda^{\sf fr}_{A,U} = \verz{\phi \mid \text{ for all $M: A\rhd U$ and $\sigma$, we have }A \vdash \phi^{(\sigma,M)}}$.\\
 In other words,  $\Lambda^{\sf fr}_{A,U} = \bigcap_{M:A\rhd U}\Lambda^{\sf fr}_M$. 
 \end{itemize}
 
 We consider one further notion in Appendix~\ref{funo}.
 
 \subsection{Motivating Remarks}\label{moresmurf}
Let us first think about ordinary provability logic.
What is the provability logic of a given theory $V$?
The arithmetisation of provability is provided by some base-theory {\sf B}.
Let us say this is $\opr_\alpha B$, where $\alpha$ is a suitable presentation of the
axioms of $V$. The base theory is `in' $V$ via an interpretation $K:V \rhd {\sf B}$.
So, $V$-provability gets the form $\opr^K_\alpha B$  in $V$.

If we switch to interpreter logic, the idea is precisely the same: the necessity operator gets
the form $\ppt K A B$, for main theory $A$.
We note that here we have $K: A \rhd {\sf B}$, on the one hand,  and that, on the other hand, the $\ccd AB$ are defined using interpretations
of $(A \wedge B)$ in finite extensions of  {\sf B}. So, both an interpretation \emph{of {\sf B} in $A$}  and interpretations of extensions of
$A$ \emph{in} extensions of {\sf B} play a role.

Both in the case of ordinary provability logic and of interpreter logic, we can quantify out the
interpretations of the base theory {\sf B} in the main theory $V$, resp. $A$.
This leads to the frame provability/interpreter logic. A frame is the pair $\tupel{A,U}$ with
$A \rhd U$. So, we abstract away from the specific interpretation.
The ordinary provability logic of a frame
was studied in \cite{viss:arit15}. In the present paper, we will consider the interpreter logic of a frame.

 \subsection{The L\"ob Conditions}
 
 We will show that every $\Lambda^{\sf fr}_K$ satisfies the L\"ob Conditions, in other words,
 it is a normal modal logic extending {\sf K}4.
 
 We will first verify a set of conditions that are equivalent to the L\"ob Conditions and then
 prove the equivalence.
  
  \begin{lemma}\label{loebsmoerf}
  Let $K:U\lhd A$ be an FM-interpretation. We write $\graydi := \graydi_{K,A}$.
We have:
\begin{enumerate}[a.]
\item
$A \vdash B\to C$ implies $A \vdash \graydi B \to \graydi C$ and $A \vdash \graysq B \to \graysq C$. 
\item
$A \vdash \neg\graydi\bot$ and $A \vdash \graysq \top$.  
\item
$A \vdash \graydi(B\vee C) \iff (\graydi B \vee \graydi C)$ and $A \vdash \graysq(B\wedge C) \iff (\graysq B \wedge \graysq C)$
\item
 $A \vdash \graydi\graydi B \to \graydi B$ and
 $A \vdash \graysq B \to \graysq\graysq B$.
\end{enumerate}
\end{lemma}

\begin{proof}
Ad (a). Suppose $A \vdash B\to C$. Then, \[(U + \ccd AB) \rhd (A\wedge B)  \rhd (A\wedge C).\]
So, $U + \ccd AB \vdash \ccd AC$, and, thus, $A \vdash \graydi B \to \graydi C$.

We note that $\neg\graysq\neg B$ is $\neg\neg \graydi \neg\neg B$. So, by the first conjunct of (a), we have
$U \vdash \graydi B \iff \neg\graysq\neg B$. As a consequence, we may switch between $\graydi$-versions of principles
 and their dual $\graysq$-formulations
 in a confident way. Thus, we omit the verification of the second conjuncts of (a-d).

Principles (b) and (c) are immediate from Theorem~\ref{entwederodersmurf}. 

Ad (d).
We have \[  (U+ \ccd A{\graydi B}) \rhd (A + \graydi B) \rhd (U+ \ccd AB) \rhd (A \wedge B).\]
So, $U+ \ccd A {\graydi B} \vdash \ccd AB$. If follows that $A \vdash \graydi\graydi B \to \graydi B$.
\end{proof}
 
 \begin{theorem}\label{lobsmurf}
 Let $K:U\lhd A$ be an FM-interpretation. Then, $\Lambda^{\sf fr}_{K}$ is a normal modal logic extending
 {\sf K4}, in other words, $\Lambda^{\sf fr}_{K}$ satisfies the L\"ob Conditions.
 \end{theorem} 
 
 \begin{proof}
 It is sufficient to show that (a-d) of Lemma~\ref{loebsmoerf} imply the L\"ob Conditions.
 clearly, {\sf L}1 or necessitation follows from (a,b). 
 
 We verify {\sf L}2. We have $A \vdash (B \wedge (B\to C)) \to C$. Ergo, by (a), we find
 $A \vdash \graysq(B \wedge (B\to C)) \to \graysq C$. Applying the second conjunct of (c), we obtain
  $A \vdash (\graysq B \wedge \graysq(B\to C)) \to \graysq C$.
  
  Finally, {\sf L}3 is identical to (d).
 \end{proof}
 
 \noindent
 We note that, conversely, the L\"ob Conditions imply (a-d) of Lemma~\ref{loebsmoerf}.
 
We will see that it is possible to get as extensions L\"ob's Logic and {\sf K}45 and {\sf S}4. 

\subsection{L\"ob's Logic}
In what circumstances we do have interpreter logics that extend L\"ob's Logic {\sf GL}? We provide a
basic result concerning that question.

We say that theories $U$ and $V$ are \emph{reconcilable} iff there are consistent finite extensions-in-the-same-language
$U'\supseteq U$ and $V' \supseteq V$, such that $U'\mutint V'$. The theories $U$ and $V$ are \emph{irreconcilable} iff they are
not reconcilable.

The simplest case of irreconcilability is when one of $U$, $V$ is inconsistent.
The second simplest case is when one of $U$, $V$ is RE and essentially undecidable and the
other is decidable and complete. Theorem~\ref{smartsmurf} will tell us that, if  one of $A$, $U$  of an FR-frame $\tupel{A,U}$ is sequential, then 
$A$ and $U$ are irreconcilable. The next theorem directly connects irreconcilability to L\"ob's Principle. 
\begin{theorem}\label{whenloeb}
Consider an FM-frame $\tupel{A,U}$.
 Then,  $A$ and $U$ are irreconcilable iff $\Lambda^{\sf fr}_{A,U}$ extends L\"ob's Logic. 
\end{theorem}

\begin{proof}
Let $\tupel{A,U}$ be an FM-frame.

\emph{We prove the left-to-right direction.} Suppose $A$ and $U$ are irreconcilable. 
 Consider any $K$ such that $K:A\rhd U$.
It is sufficient to show that $\Lambda^{\sf fr}_K$ extends L\"ob's Logic.
It is well-known that, over {\sf K}4, L\"ob's Principle and L\"ob's Rule are equivalent.\footnote{See \cite[Chapter 1, Section 1, Exercise 5(iii), p75]{smor:self85}.
Smory\'nski attributes the result to Macintyre and Simmons. Alternatively, see \cite[Chapter 3, p59]{bool:logi93}.}
So, it suffices to prove closure of  $\Lambda^{\sf fr}_K$ under L\"ob's Rule.
 Suppose $A \vdash B \to  \graydi_{K,A}B$. Then,
\[ (A+B) \rhd (U+\ccd AB) \rhd (A+B).\]
So,  $(A+B) \mutint  (U+\ccd AB)$ and, thus, by
irreconcilability, $A+B$ is inconsistent, i.e., $A \vdash \neg\, B$.

\emph{We prove the right-to left direction.} Suppose $A$ and $U$ are reconcilable. Suppose
$(A+B) \mutint (U+C)$, where $A+B$ is consistent. 

Suppose that  $K_0: A \rhd U$ and $K_1: (A+B) \rhd (U+C)$.
We define $K:= K_1\tupel{B}K_0$, i.e., the interpretation that
is $K_1$ if $B$ and $K_0$ otherwise. Clearly, $K:A \rhd U$ and $K': (A+B) \rhd (U+C)$, where $K'$ has the same
underlying translation as $K$.
Since $(U+C) \rhd (A+B)$,
we have $U+C \vdash \ccd AB$ and, so, $A+B \vdash \graydi_{K,A}B$. We also have $A\nvdash \neg \, B$.
So, $A$ is not closed under L\"ob's Rule.
We have: \[A \nvdash \graysq_{K,A}(B \to \graydi_{K,A}B) \to \graysq_{K,A} \neg\, B,\] by the fact that
 L\"ob's Rule follows  from L\"ob's Principle.
So, we do not have 
L\"ob's Principle in $\Lambda_K^{\sf fr}$.\footnote{We note that the detour over L\"ob's Principle is necessary here.
Closure of $A$ under a rule implies closure of the associated logic under the same rule, but not \emph{vice versa}.
So, we need to show that a principle has a counter-instance.}
\end{proof}

\begin{remark}
We note that Theorem~\ref{whenloeb} is a correspondence result for L\"ob's Principle.
Here the FM-frame $\tupel{A,U}$, is the analogue of a Kripke frame. The interpretation $K:A \rhd U$ in combination with a mapping $\sigma$ from
the propositional atoms to the sentences of the language of $A$ is the analogue of a 
model on the frame.
\end{remark}

When we have L\"ob's Principle, we also have  an analogue of Feferman's Theorem of the interpretability of inconsistency.
\begin{theorem}\label{fefersmurf}
  Suppose $\tupel{A,U}$ is an FM-frame and $A$ and $U$ are irreconcilable. We have:
\begin{enumerate}[a.]
\item
Suppose $K:A \rhd U$. Then $A \rhd (A+\graysq_{K,A} \bot)$.
\item
Suppose $M:B \rhd U$ and $A \rhd B$. Then $A \rhd (B+\graysq_{M,A} \bot)$.
\end{enumerate}
\end{theorem}

\begin{proof}
Ad (a). We have:
\begin{eqnarray*}
 (A+\graydi_{K,A} \top)  & \vdash & (A+\graydi_{K,A} \graysq_{K,A} \bot) \\
 & \rhd & (A+ \graysq_{K,A} \bot)
 \end{eqnarray*}
Trivially, $(A+ \graysq_{K,A} \bot) \rhd (A+ \graysq_{K,A} \bot)$. So, by a disjunctive interpretation, we find
$A \rhd (A+\graysq_{K,A} \bot)$.

Ad (b). Suppose $P:A \rhd B$. By (a), we have:\qedright
\[A \rhd (A+ \graysq_{M \circ P, A} \bot) \rhd (B+\graysq_{M,A} \bot).\]
\end{proof}

\subsection{{\sf S}4}
An FM-interpretation $K:A\rhd U$ is \emph{companionable} iff, for every $B$ of the $A$-language, there is a $C$ of the $U$-language such that
$(A+B) \mutint (U+C)$, where the interpretation of $U+C$ in $A+B$ has the same underlying translation as $K$. 

We can define  companionship in terms of the category $\mathbb E$ enriched by designated arrows
for finite extensions as indicated in the diagram below.

 \[
 \begin{tikzcd}
 U+C \arrow[dotted,leftarrow]{r}{M} & A+B \\
 U+C \arrow[dotted]{r}{K'} & A+B \\
 U \rar{K}\arrow[dotted]{u}{\subseteq} & A \arrow[swap]{u}{\subseteq} 
 \end{tikzcd}
 \]
 
 \medskip\noindent
 Here we  require no commutation for $M$.

\begin{theorem}\label{wijsneussmurf}
Consider an FM-interpretation $K:A \rhd U$. Then, $K$ is companionable iff $\Lambda_K^{\sf fr}$ extends {\sf S}4.
\end{theorem}

\begin{proof}
Suppose $K$ is companionable. Consider any $B$ and suppose $K': (A+B) \rhd (U+C)$, where $K'$ is based on the 
same translation as $K$ and that $(U+C) \rhd (A+B)$. It follows that $U+C \vdash \ccd AB$.
Hence, $A+B \vdash \graydi_{K,A}B$.

Conversely, suppose $\Lambda_K^{\sf fr}$ contains the reflection principle, aka {\sf M}.
We have that $A+B \vdash \graydi_{K,A} B$, so, there is an interpretation
$K'$ based on the same translation as $K$, such that $K':(A+B) \rhd (U+ \graydi_{K,A}B)$.
Conversely, $(U+ \graydi_{K,A}B) \rhd (A+B)$.
\end{proof}

\begin{remark}
We note that the characterisation provided by Theorem~\ref{wijsneussmurf} is rather different in nature
from the one given of L\"ob's Principle in Theorem~\ref{whenloeb}. First, in Theorem~\ref{wijsneussmurf},
we consider a property of interpretations rather than a property of frames as in Theorem~\ref{whenloeb}. Secondly, we use more notions
to formulate companionship than for irreconcilability. In Appendix~\ref{funo}, we prove a result for the reflection principle that is more
in the spirit of Theorem~\ref{wijsneussmurf}. However, to do that we need to consider local interpreter logics of an FM-frame rather than
the unique (global) interpreter logic of the frame.
\end{remark}

\begin{example}\label{losersmurf}
Let $A$ be e.g. the theory of the ordering of the natural numbers. The theory $A$ is finitely axiomatisable.
Theorem~\ref{smurferella} will tell us that $A$ is Friedman-reflexive. The identical interpretation ${\sf Id}_A$ of
$A$ in itself is clearly companionable. So, $\Lambda^{\sf fr}_{{\sf Id}_A}$ extends ${\sf S}_4$. 

In fact, we can show that the modality trivialises for this example. Consider any $B$ in the $A$-language.
The sentence  $\graydi_{{\sf Id}_A,A} B$ is either provable or refutable in $A$. If it is refutable,
we have $A \vdash \graydi_{{\sf Id}_A,A} B \to B$. Suppose it is provable. So,
$A \vdash \graydi_{{\sf Id}_A,A} B$. It follows that $A \rhd (A+B)$. So, $A+B$ is consistent, and, hence,
$B$ is provable in $A$. Thus, $A \vdash \graydi_{{\sf Id}_A,A} B \to B$. So, in both cases,
we have $A \vdash \graydi_{{\sf Id}_A,A} B \to B$. 
\end{example}

\begin{question}\label{q1A}
Can we find a more inspiring example of a theory with logic {\sf S}4 than Example~\ref{losersmurf}?
Is it perhaps possible to find an FM-interpretation with interpreter logic precisely {\sf S}4?
\end{question}

We remind the reader that, in a category, a morphism $a\stackarrow f b$ is a \emph{retraction} or \emph{split epimorphism} iff
there is a $b\stackarrow g a$, such that $f \circ g = {\sf id}_b$. Here $g$ is called \emph{section}, \emph{co-retraction}, or
\emph{split monomorphism}.

\begin{corollary}
Suppose the FM-interpretation $U\stackarrow K A$ is a retraction in $\mathbb E$. Then, $K$ is companionable and, hence, $\Lambda_K^{\sf fr}$ extends {\sf S}4.
\end{corollary}

\begin{proof}
Suppose  $U\stackarrow K A$ is an FM-interpretation which is a retraction in $\mathbb E$. Let $M$ be the corresponding section, i.e.,
 $K \circ M = {\sf Id}_A$.
Consider any $B$ in the $A$-language. We have: $(U+B^M) \rhd (A+B)$. 
Moreover, writing $\equiv$ for having the same theorems, we have:
\begin{eqnarray*}
 (U+B^M)   & \stackarrow{K'} & (A+B^{MK})\\
 & \equiv &  (A+B)
\end{eqnarray*}

\noindent
Here $K': (A+B) \rhd (U+B^M)$ has the same underlying translation as $K$. 
\end{proof}
 
 \subsection{Relations between Logics}
 The notion of sameness of theories that is relevant in the present paper is
 sentential congruence or $\mathbb E$-isomorphism (see Appendix~\ref{notarissmurf}, for more on these notions).
 In the case of interpreter logics, the relevant notion of sameness of interpretations is as follows.
 Suppose $V_0 \stackarrow{K_0} W_0$, $V_1 \stackarrow{K_1} W_1$.
 \begin{itemize}
 \item
$K_0 \approx K_1$ iff, there are $V_0\stackarrow{M} V_1$, 
 $V_1\stackarrow{\breve M} V_0$, $W_0\stackarrow{P} W_1$, $W_1\stackarrow{\breve P} W_0$,
 such that $M,\breve M$ and $P,\breve P$ are pairs of inverses in $\mathbb E$ and 
$ K_1\circ M = P \circ K_0$ in $\mathbb E$.
 \end{itemize}
 
 \[
 \begin{tikzcd}
 V_0 \rar{K_0}\arrow[leftrightarrow]{d}[swap]{M, \breve M} & W_0 \arrow[leftrightarrow]{d}{P,\breve P} \\
 V_1 \rar{K_1} & W_1
 \end{tikzcd}
 \]
 
 \noindent We note that $\approx$ is simply isomorphism in the arrow category ${\sf Arr}(\mathbb E)$.
 
 We have the following theorem.
 \begin{theorem}\label{gelijkheidssmurf}
 Suppose $U_0$ is Friedman-reflexive and $A_0$, $A_1$ are finitely axiomatised.
 Suppose further that $U_0\stackarrow{K_0}A_0$, $U_1\stackarrow{K_1}A_1$ and
 $K_0 \approx K_1$. Then $U_1$ is Friedman-reflexive and $\Lambda^{\sf fr}_{K_0} = \Lambda^{\sf fr}_{K_1}$. 
 \end{theorem}
 
 \noindent
 The theorem is one of these examples where the truth is immediately clear but it still requires some work to
 really prove it. We give the proof in Appendix~\ref{verismurf}.
 
 We have a second preservation theorem.
 \begin{theorem}\label{machtigesmurf}
 Let $F$ be an endofunctor $F$ of $\mathbb D$. Here we suppose that $F$ is specified on concrete theories and interpretations.
  We assume that:
\begin{itemize}
\item
$F$ preserves finite axiomatisability.
\item
 For each theory $V$, there is a faithful interpretation  $V \stackarrow{\eta_V} F(V)$.
 \item
 Suppose $\Gamma$ is a set of sentences in the $V$-language.
 Then, $F(V+\Gamma) = F(V) + \Gamma^{\eta_V}$.\footnote{This condition does not
 look very `categorical', since it goes into the hardware. In Appendix~\ref{verismurf}, we 
discuss a more categorical formulation.}
\end{itemize}
Let $U$ be Friedman-reflexive and suppose $F(V) \stackarrow{M_V} V$, for all extensions
$V$ of $U$ in the same language. Here there is no further constraint on the $M_V$.
Let $A$ be finitely axiomatised and suppose  $U \stackarrow{K} A$.

We have: $\Lambda^{\sf fr}_{\eta_A \circ K} \subseteq \Lambda^{\sf fr}_K$.
\end{theorem}

\noindent
We give the proof in Appendix~\ref{verismurf}.

 Let ${\sf in}^{V_0,V_1}_i$ be the obvious interpretation of $V_i$ in $V_0\owedge V_1$.

\begin{corollary}\label{insmurf}
Suppose $K: A \rhd U$ is an FM-interpretation.
Then, we have $\Lambda^{\sf fr}_{{\sf in}_0^{A,B} \circ K} \subseteq \Lambda^{\sf fr}_K$.
\end{corollary}

\noindent
We will see a further application of Theorem~\ref{machtigesmurf} in Subsection~\ref{wijsneuzigesmurf}.
 
 We end with two simple observations. 
 
 \begin{theorem}
  Suppose $U$ is Friedman-reflexive and $K:U \lhd A$. Let $B$ be a sentence in the $A$-language.
 Then, $A \vdash \graydi_{K,A}(B\wedge C) \iff \graydi_{K,A+B}C$.
 \end{theorem}
  
 Suppose $K:A \rhd U$.
 Let ${\sf Th}(K) := \verz{B \mid A \vdash B^K}$.
 
 \begin{theorem}
 Suppose $U$ is Friedman-reflexive and $K:U \lhd A$.
 Let $K^\ast: {\sf Th}(K) \lhd A$ be the interpretation based on $\tau_K$ the translation given with $K$. Then,
 $\Lambda^{\sf fr}_K = \Lambda^{\sf fr}_{K^\ast}$.
 \end{theorem}

 \noindent
 We leave the proofs to the reader. 
 
 \begin{question}\label{q2}
 Suppose $K,M:U \lhd A$ are FM-interpretations and  ${\sf Th}(K)={\sf Th}(M)$. Do we have  $\Lambda^{\sf fr}_K = \Lambda^{\sf fr}_{M}$,
 or is there a counter-example?
 \end{question}
 
 In the rest of this paper, we will have a closer look at interpreter logics
 over certain special classes of base theories.

 \section{Complete Theories}\label{compunde}

In this section, we discuss complete theories.

\begin{theorem}\label{smurferella}
Suppose $U$ is a complete theory. Then, $U$ is Friedman-reflexive.
\end{theorem}

\begin{proof}
Suppose $U$ is complete. So, modulo $U$-provable equivalence, we only
have propositions $\top$ and $\bot$. So, $A$ has as pro-interpreters, modulo provable equivalence,
either just $\bot$ or both $\bot$ and $\top$. In the first case, the desired interpreter
is $\bot$, in the second it is $\top$.
\end{proof}
 
  \begin{example}
  Presburger Arithmetic is complete and decidable. So, it is Friedman-reflexive but not effectively so.
  
  True arithmetic ${\sf Th}(\mathbb N)$ extends {\sf PA} and is, hence, effectively Friedman-reflexive. Thus, we have an example
 of a complete theory that is effectively Friedman-reflexive. We note that the \ccz\ that takes $\top$ and $\bot$ as values 
 cannot be recursive, providing an example of a salient non-recursive choice of \ccz.
 \end{example}
 
 We consider the interpreter logic for a complete base. We note that if 
 $\tupel{U,A}$ is an FM-frame and $U$ is complete, then $U$ must be decidable and, hence
 any choice for \ccz\ must be non-computable.
 
 We show that the interpreter logics for complete bases extend {\sf K}45.
 
 \begin{theorem}
 Suppose $K:A\rhd U$ is an FM-interpretation and $U$ is complete. Then, $A \vdash \graydi_{K,A}B \to \graysq_{K,A}\graydi_{K,A}B$. 
 \end{theorem}
 
 \begin{proof}
  Suppose $K:A\rhd U$ is an FM-interpretation and $U$ is complete. Consider any $A$-sentence $B$.
  If $A+ \graydi_{K,A}B$ is inconsistent, we are done. If not, it follows that $U+\ccd AB$ is consistent.
  Hence, by completeness, $U \vdash \ccd AB$. So, $A \vdash \graydi_{K,A}B$ and, hence,
 $ A \vdash \graysq_{K,A}\graydi_{K,A}B$. We may conclude that in both cases, we have
  $ A \vdash\graydi_{K,A}B \to  \graysq_{K,A}\graydi_{K,A}B$.
 \end{proof}
  
 We note that, if $U \nrhd A$, then $U \nvdash \ccu A$ and, so, $U \vdash \neg\, \ccu A$.
 It follows that $A \vdash \graysq_{K,A}\bot$. So, the interpreter logic for $A$ trivialises.
 Suppose, on the other hand, that $U \rhd A$. Then, $A \vdash \graydi_{K,A} \top$.
 It follows that $U$ is mutually interpretable with a finite sub-theory.
 
 We suspect that most non-finitely axiomatisable complete and decidable theories in the literature
 have the property that they
  are \emph{not} interpretable in a finite sub-theory. This has been verified for Presburger
 Arithmetic. See \cite{pakh:mult20b}. (So, the interpreter frame logic of Presburger Arithmetic
 trivialises.)
 
 There are, of course, examples of consistent, finitely axiomatised, 
 complete and decidable theories like the theory of dense linear orderings without end-points and the
 theory of the ordering of the natural numbers.
  
 \begin{question}\label{q2A}
 Is there an example of an FM-interpretation $K:A\rhd U$, where $U$ is complete, with an interesting interpreter logic?
 \end{question}
 
 \section{Finitely Axiomatised Theories} \label{finax}
If the Friedman-reflexive base is finitely axiomatised, we 
 can view the embedding functor as an embedding of the finite extensions of the base $A$ in the finitely axiomatised theories.
 So, we can view the Friedman-reflexivity of the base as the existence of an adjoint of this functor.
 
 There are plenty of consistent complete finitely axiomatised theories, so we do not lack
 examples of the phenomenon of a consistent finitely axiomatised Friedman-reflexive theory.
 However, these examples cannot be effectively Friedman-reflexive, since they, clearly, cannot
 be essentially undecidable. 
 
  \begin{question}\label{q3}
 Is there a consistent finitely axiomatised theory that is effectively Friedman-reflexive?
 \end{question}
 
 In the case of a finitely axiomatised base, there is, of course, the salient interpreter logic of
 $A$ over $A$ via the identical interpretation. This logic satisfies {\sf S}4, since ${\sf Id}_A$ is
 clearly companiable.
  
 \section{Essentially Sententially Reflexive Theories}
 In this section we study essential sentential reflexiveness.
 A theory $U$ is \emph{essentially sententially reflexive} if, for some $N$, we have $N:U \rhd {\sf S}^1_2$ and, for all
 $U$-sentences $A$,
 $U \vdash \opr^N_n A \to A$. Here $A$ ranges over $U$-sentences and $\opr_n$ means provability in
 predicate logic using only involving formulas with depth of quantifier alternations $\leq n$.  As a default
 we assume in our notation that $n$ exceeds $\rho(A)$, the  depth of quantifier alternations of $A$.
 
 We will write $\apr_AB$ for $\opr_{\rho(A \to B),A}B$ and $\aco_AB$ for $\oco_{\rho(A \wedge B),A}B$.
 
 \subsection{A Basic Fact}
 We have the following theorems. The second result provides a coordinate-free characterisation of Essentially Sententially Reflexive Theories
 in the sequential case.
 
 \begin{theorem}\label{essereuno}
Suppose $U$ is essentially sententially reflexive. Then, there is an $N:U \rhd {\sf S}^1_2$, such that,
for all $\Sigma^0_1$-sentences $S$ and for all $M:U \rhd {\sf S}^1_2$, we have $U \vdash S^N \to S^M$.
 \end{theorem}
 
 \begin{proof}
Suppose that $U$ is essentially sententially reflexive with witness $N_0:U \rhd {\sf S}^1_2$. We start with the observation that, for all $U$-sentences
$A$ and for all $m \geq \rho(A)$, we have
$U \vdash \opr^{N_0}_m A \to A$. This follows immediately by replacing $A$ by $(A\wedge B)$, where $B$ is a tautology with $\rho(B) = m$.
Let $N $ be a  logarithmic cut of $N_0$. Consider any $M: U \rhd {\sf S}^1_2$. We have, for a sufficiently large $m$ and for some
$U$-theorem $D$, that
$U \vdash S^N \to \opr^{N_0}_{m, D} S^M$ (see Theorem~\ref{sicosmurf}). Here we can take $D := (E_{{\sf S}^1_2} \wedge \bigwedge {\sf S}^1_2)$. 
It follows that $U \vdash  S^N \to S^M$.
\end{proof}

\begin{theorem}\label{essere}
Suppose $U$ is sequential. Then the following conditions are equivalent.
\begin{enumerate}[a.]
\item
$U$ is essentially sententially reflexive.
\item
There is an $N:U \rhd {\sf S}^1_2$, such that
 for all $\Sigma^0_1$-sentences $S$ and all $M:U \rhd {\sf S}^1_2$, we have $U \vdash S^N \to S^M$.
 \item
 Consider any $N^\ast:U \rhd {\sf S}^1_2$. There is an $N^\ast$-cut $I$, such that, for all 
  $\Sigma^0_1$-sentences $S$ and all $N^\ast$-cuts $J$, we have  $U \vdash S^I \to S^J$.
\end{enumerate}
\end{theorem}

\begin{proof}
Suppose $U$ is sequential. 

Theorem~\ref{essereuno} tells us that (a) implies (b). We prove the other direction.
Suppose $N$ witnesses (b). Consider any  $U$-sentence $A$.
Since $U$ is sequential, there is an $N$-cut $I$, such that $U \vdash \apr^IA \to A$. Since
$U \vdash \apr^N A \to \apr^I A$, we find $U \vdash \apr^NA \to A$.  So $N$ witnesses the
essential sentential reflexivity of $U$.

We prove the implication from (b) to (c). Let $N$ witness (b). Consider any $N^\ast:U \rhd {\sf S}^1_2$.
By a result of Pudl\'ak, there is an $N$-cut $I$ and an $N^\ast$-cut $I^\ast$, such that $I$ and $I^\ast$ are
$U$-definably, $U$-provably isomorphic. We take $I^\ast$ as our witness for (c). Reason in $U$. Let $J$ be any
$N^\ast$-cut.
Suppose $S^{I^\ast}$. Then, $S^I$, and, hence $S^N$. It follows that $S^J$.

We prove the implication from (c) to (b). We take as witness for (a), the $N^\ast$-cut $I$ promised by (c).
Consider any $M:U \rhd {\sf S}^1_2$. Let $J$ and $J^\ast$ be $U$-definably, $U$-provably isomorphic cuts
of $M$ and $I$. Reason in $U$. Suppose $S^I$. Then, $S^{J^\ast}$. So, $S^J$, and, hence, $S^M$.  
\end{proof}

 \subsection{Essential Sentential Reflexiveness implies Friedman-reflexiveness}
 We have the following theorem.
  
 \begin{theorem}\label{erfr}
 Suppose $U$ is essentially sententially reflexive with witnessing interpretation $N$. Then,
 $\ccu A := \frcr{A}{N}$ witnesses that $U$ is effectively Friedman-reflexive.
 \end{theorem}
 
  \noindent The argument is, of course, just Harvey Friedman's argument for the case of {\sf PA}.
 
  \begin{proof}
  Suppose $N$ witnesses the sentential essential reflexiveness of
$U$.

 Suppose $(U+B) \rhd A$. Then, for some finite sub-theory $D$ of $U$, we have that $(D+B)\rhd A$. If follows that
 $U \vdash ((D+B) \rhd A)^N$ and, so, $U \vdash \oco^N_{m,D}B \to \frcr AN$, for sufficiently large $m$. 
 By the  sentential essential reflexiveness of
$U$, it follows that $U+B \vdash  \frcr{A}{N}$.
 
 For the opposite direction, suppose $U+B \vdash \frcr{A}{N}$. Then, by the Interpretation Existence Lemma, we find $(U+B)\rhd A$. 
 \end{proof}

\subsection{Peano Corto}\label{cortoboven}
In this subsection, we discuss  the sententially essentially reflexive theory $\paco$. 
However, since we mainly want to illustrate the idea of a weak  sententially essentially reflexive theory, it
seemed good to zoom in on one.

In our paper \cite{viss:pean14}, we used ${\sf PA}^-$ as starting point.
Here,  we use ${\sf S}^1_2$ as starting point since it fits the set-up of the present paper better. 

We introduce $\paco$ and its little brother $\paoo$ and its
big brother $\paba$.

\begin{itemize}
\item
Peanissimo or $\paoo$ is the theory
\[{\sf S}^1_2 + \verz{(S \to S^I) \mid \text{$S$ is an $\exists\Sigma^{\sf b}_1$-sentence and
  $I$ is an ${\sf S}^1_2$-cut}}.\] 
  This theory is identical to ${\sf S}^1_2+ \verz{\apr A \to A \mid \text{$A$ is an arithmetical sentence}}$.
  \item
 Peano Corto or $\paco$ is the theory 
\[{\sf S}^1_2 + \verz{(S \to S^I) \mid \text{$S$ is a $\Sigma^0_1$-sentence and
  $I$ is an ${\sf S}^1_2$-cut}}.\] 
  Here $\Sigma^0_1$ means a block of existential quantifiers followed by a $\Delta_0$-formula.
  In fact, Peano Corto is Peanissimo plus the scheme \[S \to \exists x\,\exists z\, (2^x = z \wedge S_0(x)),\] where
  $S = \exists x\, S_0(x)$.
   \item
   Peano Basso or $\paba$ is the theory \[{\sf S}^1_2 + \verz{(S \to S^I) \mid \text{$S$ is a $\Sigma^0_{1,\infty}$-sentence and
  $I$ is an ${\sf S}^1_2$-cut}}.\] 
  Here $\Sigma^0_{1,\infty}$ is the class of formulas given by a block of existential quantifiers and bounded 
  universal quantifiers, where both sorts may occur in an alternating way, followed by a $\Delta_0$-formula.
  In  \cite{viss:pean14}, it is shown that Peano Basso is Peano Corto plus $\Sigma^0_1$-collection.
   \end{itemize}
   
    In  \cite{viss:pean14}, it is shown that Peano Corto and Peano Basso are essentially sententially reflexive w.r.t.\ the identical cut.
    A similar argument shows the same for Peanissimo.
    
   All three theories are locally cut-interpretable in ${\sf S}^1_2$, i.o.w., ${\sf S}^1_2 \rhd_{\sf cut,loc} \paoo$ and ${\sf S}^1_2 \rhd_{\sf cut,loc} \paco$
   and ${\sf S}^1_2 \rhd_{\sf cut,loc} \paba$. Also all three theories are mutually cut-interpretable.
  
  We remind the reader that each theory is recursively axiomatisable, since we can replace the
  cuts $I$ in our formulation by $E\tupel{{\sf cut}_x(E)}(x=x)$, where $E$ ranges over formulas with at most
  the free variable $x$. Here  ${\sf cut}_x(E)$ is the ${\sf S}^1_2$-sentence that expresses 
  `$\verz{x\mid E(x)}$ is a cut' and $F\tupel{G}H$ is $((G\to F) \wedge (\neg\, G \to H))$. 
  
  Since, the identical cut is the designated cut, we can, by Theorem~\ref{erfr}, take $\ccu A := \frc{A}$ in each of our theories.
 
 \begin{theorem}\label{pichar}
 Suppose $P$ is a $\Pi^0_1$ interpreter of $A$ over Peano Corto.
 Then, ${\sf EA} \vdash P \iff \frc{A}$.
 \end{theorem}

  \begin{proof}
  Suppose $P$ is a $\Pi^0_1$ interpreter of $A$ over $\paco$.
  Then, by uniqueness, we have $\paco \vdash P \iff \frc{A}$.
  So, for some ${\sf S}^1_2$-cut $J$, we have
  ${\sf S}^1_2 \vdash (P \iff \frc{A})^{J}$. Hence, by a meta-theorem of Paris and Wilkie,
 we have ${\sf EA} \vdash P \iff \frc{A}$. See \cite{wilk:sche87} and \cite{viss:insi92}.
 \end{proof}
   
 So, we have characterised $\frc{A}$ as a $\Pi^0_1$-sentence up to {\sf EA}-provable equivalence in
 a coordinate-free way. This improves the result of \cite{viss:seco11}, where this was only done for
 finitely axiomatised \emph{sequential} theories.
 
 We have a version of the Friedman characterisation over Peano Corto. 
 
 \begin{theorem}\label{frigar}
 Suppose $A$ is sequential. Then, 
 \[A\rhd B \;\;\text{ iff }\;\;\paco \vdash \frc{A} \to  \frc{B}.\]
 \end{theorem} 
 
 \begin{proof}
 The left-to-right direction is just Theorem~\ref{minismurf} in combination with the fact that we can take
  $\ccu A := \frc{A}$ over Peano Corto.
  
  From right to left: suppose $\paco \vdash \frc{A} \to  \frc{B}$.
  Then, ${\sf S}^1_2 \vdash  (\frc{A} \to  \frc{B})^I$, for some cut $I$.
  So, ${\sf S}^1_2 \vdash  \frc{A} \to  \frcr{B}{I}$. We have:
  \[ A \rhd ({\sf S}^1_2 +  \frc{A}) \rhd   ({\sf S}^1_2 + \frc{B}) \rhd B.\]
  The first step uses the sequentiality of $A$.
 \end{proof}
 
 \begin{remark}
 The original Friedman characterisation had {\sf EA} in place of Peano Corto. A  result due to 
 Paris and Wilkie (see \cite{wilk:sche87} and \cite{viss:insi92}) shows that we have, for $\Pi^0_1$-sentences $P$ and $Q$:
 \[ \paco \vdash P \to Q \;\;\;\Iff \;\;\; {\sf EA} \vdash P \to Q.\]
 So, the connection between the two results is obvious.
 However, the internal version of the characterisation in {\sf EA} needs cut-free {\sf EA}-provability and
 ordinary $\paco$-provability.
 \end{remark}
 
 \begin{remark}\label{cattwee}
The theory ${\sf Seq}(V)$ is specified as follows. We add a unary predicate $\dom$ and a binary predicate
 $\in$ to the signature of $V$, we relativise $V$ to $\dom$ and we add the (unrelativised) axioms for Adjunctive Set Theory {\sf AS}
 plus an axiom that states that every element of $\dom$ is an empty set. We can show that {\sf Seq} supports
 a functor from $\mathbb D$ to  $\mathbb D_{\sf seq}$ and from  $\mathbb D_{\sf fin}$ 
 to  $\mathbb D_{\sf fin,seq}$. See Appendix~\ref{notarissmurf} for details.
It is easily seen that we have ${\sf Seq}(A) \mutint ({\sf S}^1_2+ \frc A)$.

We see that we can split the functor $H$ of Remark~\ref{cateen} in two stages. First, we have a projection $\pi$ of
$\mathbb D_{\sf fin}$ to $\mathbb D_{\sf fin,seq}$. This can be either $A \mapsto {\sf Seq}(A)$ or
$A \mapsto ({\sf S}^1_2+ \frc A)$. Then, we have a (lax)  embedding of $\mathbb D_{\sf fin,seq}$
into ${\mathbb B}_{\paco}$.
 \end{remark}
  
Suppose $K:A \rhd \paco$.
Since, we can apply the G\"odel Fixed Point Theorem in the usual way because \ccz\ can be represented by a predicate, 
we have L\"ob's Logic.
This also follows from Corollary~\ref{tuinmansmurf} in combination with the fact that $\paco$ is sequential.
That theorem, however, has a much more involved proof.

If $K$ is $\Sigma^0_1$-sound,  $\Lambda_K^{\sf fr}$ is
 precisely L\"ob's Logic. In the case of $\paco$  we can verify Solovay's Theorem simply using Solovay's proof.
 The reason is that $\paco$ proves that
 $\exists x,y\, (2^{2^x}=y \wedge S_0(x))$ from $\exists x\, S_0(x)$, where $S_0$ is $\Delta_0$ or $\Delta_0(\omega_1)$.
 This delivers $\Sigma^0_1$-completeness. This argument is not present for $\paoo$. However, we still have Solovay's Theorem
 for $\paoo$ as a special case of Theorem~\ref{solosmurf}.

 We can see that Peano Corto has some definite advantages over ${\sf S}^1_2$ in the role
 of base theory. We have a coordinate-free representation of the interpreter variant of provability.
 Moreover, we have the insights contained in Theorems~\ref{pichar} and \ref{frigar} and the good properties
 of the interpreter logics over Peano Corto. However, there is a down-side too.
 \begin{enumerate}[I.]
 \item
 Peano Corto is not finitely axiomatisable. 
 \item
 Peano Corto is not  interpretable in ${\sf S}^1_2$. If it were it would be mutually interpretable with
 ${\sf S}^1_2$ and this contradicts Theorem~\ref{smartsmurf} that we will prove later. In fact, no sequential Friedman-reflexive theory
 is interpretable in ${\sf S}^1_2$ by the same argument. As a consequence, there are no interpreter
 logics for ${\sf S}^1_2$ with Peano Corto as base (or, with any sequential Friedman-reflexive base). 
 \item
Even if Peano Corto is interpretable in some reasonably weak concrete $A$, like {\sf EA}, it is not always clear that
we can find an interpretation that does not involve arithmetisation. We discuss this kind of problem in Section~\ref{corema}.
 \end{enumerate}

 \section{Friedman-reflexivity meets Sequentiality}
 We already met  some specific sequential theories that are essentially sententially reflexive.
 In this section, we look at sequential theories that are Friedman-reflexive in general.
 Moreover, we look at interpreter logics for sequential $A$, also in cases where the base is not
 itself sequential.
 
 \subsection{Characterisation}
In this subsection, we provide  characterisations both of the interpreters provided by sequential Friedman-reflexive bases
and of such bases themselves.

We show that $\ccu A$ always has the form of a restricted consistency statement of $A$ on some cut.

\begin{theorem}\label{crc}
Suppose $U$ is sequential and Friedman-reflexive. Let $N:{\sf S}^1_2 \lhd U$. Then, for some $N$-cut $I$, we have
$U \vdash \ccu A \iff \frcr{A}{I}$. 

In case $U$ is RE and effectively Friedman-reflexive, we can find $I$ effectively.
\end{theorem}

\begin{proof}
We have $(U+\ccu A) \rhd A$. Let $K$ be a witnessing interpretation.
 It follows, for some $N$-cut $I$, that $(U+\ccu A) \vdash \frcr{A}{I }$.
We can see that by choosing $I$ so short that we can verify reflection for proofs
involving only formulas of $\rho$-complexity $\leq m := \rho(A)+\rho(K)$ w.r.t. a truth-predicate for formulas
of $\rho$-complexity $\leq m$. This truth-predicate 
works on an appropriate $N$-cut $I^\ast$. We choose $I$ smaller than $I^\ast$.
So, we have $U+\ccu A \vdash \frcr{A}{I}$. 

Conversely, since, $(U+\frcr{A}{I}) \rhd A$, it follows
that $U+\frcr{A}{I} \vdash \ccu A$.

Trivially, $I$ can be effectively found when $\ccu A$ is given and $U$ is RE.
\end{proof}

We provide a characterisation of Friedman-reflexivity in the sequential case.

\begin{theorem}\label{characunus}
Suppose $U$ is sequential. 
\begin{enumerate}[A.]
\item
The following are equivalent.
\begin{enumerate}[a.]
\item
$U$ is Friedman-reflexive.
\item
For all $\Sigma^0_1$-sentences $S$, there is an $N:U \rhd {\sf S}^1_2$, such that,
for all $M:U \rhd {\sf S}^1_2$, we have $U \vdash S^N \to S^M$.
\item
Consider any $N^\ast:{\sf S}^1_2 \lhd U$. Then,  for all $\Sigma^0_1$-sentences $S$,
there is an $N^\ast$-cut $I$ such that for all $N^\ast$-cuts $J$, we have $U \vdash S^I \to S^J$.
\end{enumerate}
\item
The following are equivalent.
\begin{enumerate}[a.]
\item
$U$ is effectively Friedman-reflexive.
\item
There is a recursive function $F$ such that, for all $\Sigma^0_1$-sentences $S$, we have $F(S) =N:U \rhd {\sf S}^1_2$ and,
for all $M:U \rhd {\sf S}^1_2$, we have $U \vdash S^N \to S^M$.
\item
Consider any $N^\ast:{\sf S}^1_2 \lhd U$. There is a recursive function $G$ such that,
 for all $\Sigma^0_1$-sentences $S$,
we have $G(S)= I$, where $I$ is an $N^\ast$-cut such that for all $N^\ast$-cuts $J$, we have $U \vdash S^I \to S^J$.
\end{enumerate}
\end{enumerate}
\end{theorem}

\begin{proof}
We will just prove the equivalence between (Aa) and (Ac). The equivalence
between (Ab) and (Ac) is immediate using the fact that any two interpretations of ${\sf S}^1_2$ in $U$
have  $U$-definably, $U$-verifiably isomorphic cuts. The proof of (B) is by inspection of the proof of (A).

Suppose that $U$ is Friedman-reflexive and $N^\ast:U \rhd {\sf S}^1_2$. 
Consider any $\Sigma^0_1$-sentence $S$. 
We note that $A :={\sf S}^1_2+\neg\,S$ is finitely axiomatised. 
Let $I_0$ be the $N$-cut such that $\aco^{I_0}_{{\sf S}^1_2}\neg S$ is $U$-provably equivalent to
$\ccu A$. We find $U \vdash \aco^J_{{\sf S}^1_2}\neg S\to \aco^{I_0}_{{\sf S}^1_2}\neg S$, for all
$N$-cuts $J$. Thus, 
 $U \vdash \apr^{I_0}_{{\sf S}^1_2} S\to \apr^J_{{\sf S}^1_2} S$, for all $N$-cuts $J$.
 
 Let $I$ be an $N$-cut so that $U \vdash S^I \to  \apr^{I_0}_{{\sf S}^1_2} S$: see Theorem~\ref{sicosmurf}.
 Consider any $N$-cut $J$. By sequentiality, 
 we can find a $J$-cut $J_0$ so that $U \vdash  \apr^{J_0}_{{\sf S}^1_2} S \to S^J$.
 This uses again a soundness proof involving a truth-predicate.
 Putting everything together we find:
 \begin{eqnarray*}
 U \vdash S^I & \to & \apr^{I_0}_{{\sf S}^1_2} S \\
 & \to & \apr^{J_0}_{{\sf S}^1_2} S \\
 & \to & S^J
 \end{eqnarray*}
 
 Conversely, suppose $U$ satisfies (c). 
  Let $I$ be the $N^\ast$-cut guaranteed by (c) for $S:= \apr_{A}\bot$.
  Suppose $(U+B) \rhd A$. Then, for some $N$-cut $J$, we have $U+B \vdash \frcr{A}{J}$
  and hence $U+B \vdash  \frcr{A}{I}$. The other direction is immediate by Interpretation Existence.
\end{proof}

\begin{remark}
We note that a complete and consistent sequential theory will automatically have property (c) of Theorem~\ref{characunus}.
Of course, it should, by Theorems~\ref{smurferella} and \ref{characunus}.
\end{remark}

\begin{remark}
Our result is rather robust for the precise notion of $\Sigma^0_1$ used. 
The result works both for smaller classes and for larger ones.

It works for $\Sigma_1^{\sf b}$ and even for Diophantine sentences consisting of a block of existential
quantifiers followed by an equation $t=u$.

In the other direction, the result also applies when we admit $\omega_1$-terms in our definition of
$\Sigma^0_1$. Finally, it works when we define our class $X$ as follows:
\begin{itemize}
\item
$X ::= \top \mid \bot \mid t=u \mid \neg Y \mid
(X\wedge X) \mid (X\vee X) \mid (Y \to X) \mid \forall x < t\, X \mid \exists x\, X$
\item  
$Y ::= \top \mid \bot \mid t=u \mid \neg X \mid
\,(Y\wedge Y) \mid \,(Y\vee Y) \mid \,(X \to Y) \mid \exists x < t\, Y \mid \forall x\, Y$
\end{itemize}
\end{remark}

We give a slightly modified version of our characterisation.

\begin{theorem}\label{mocha}
Suppose $U$ is sequential and let $N:U\rhd {\sf S}^1_2$. Then,
 $U$ is Friedman-reflexive iff \textup{(\dag)} for all $\Sigma^0_1$-sentences $S$, there is an $U$ sentence $A$, such that,
 for all $U$-sentences $B$, we have $U +\verz{S^I\mid \text{$I$ is an $N$-cut}} \vdash B$ iff $U+A \vdash B$.
\end{theorem}

\begin{proof}
Suppose $U$ is Friedman-reflexive. Let $I^\ast$ be the cut guaranteed  for $S$ by Theorem~\ref{characunus}(c).
Then, it is easy to see that $S^{I^\ast}$ can be chosen as our $A$ to satisfy (\dag).

Conversely, suppose (\dag). It is clear that $U+A$ proves all $S^I$. On the other hand, taking $B:=A$, we see that some finite
conjunction of the $S^I$ will imply $A$ over $U$. We now take $J$ the intersection of all cuts occurring in this finite conjunction.
We find that $U \vdash A \iff S^J$. We take $J$ as witness for satisfaction of the characterisation of Theorem~\ref{characunus}(c).
\end{proof}

We give a final version of our characterisation that is both useful and enlightening.
We need the notion of intersection of all cuts.
Consider a sequential model $\mathcal M$. We define $\mathfrak I_{\mathcal M}$ as follows.
First, we choose an internal model $\mathcal N$ of ${\sf S}^1_2$ and then we take $\mathfrak I_{\mathcal M}$ to be the
intersection of all $\mathcal M$-definable $\mathcal N$-cuts. Using elementary facts about sequentiality, one can
easily show that $\mathfrak I_{\mathcal M}$ is independent of the choice of $\mathcal N$ in the sense that
all versions are isomorphic by $\mathcal M$-definable isomorphism. Moreover, this isomorphism is unique when restricted to
the intersection. See also \cite[Section 5.1]{viss:smal19}. 

Consider a sequential theory $U$ and let $N:{\sf S}^1_2 \lhd U$. We extend the language of $U$ with a new unary predicate $\mathfrak I$
that is interpreted in each $U$-model $\mathcal M$ as $\mathfrak I_{\mathcal M}$. Here we think of $\mathfrak I_{\mathcal M}$
as given by $\mathcal N := N^{\mathcal M}$. 
Let $U^{\sf e}$ be the set of all sentences in the extended language true in all $\mathcal M, \mathfrak I_{\mathcal M}$, where
$\mathcal M$ is an $U$-model.
Let $\mathcal I_U$ be the set of arithmetical sentences $A$ such that $U^{\sf e} \vdash A^{\mathfrak I}$.

An important insight is that $\mathcal I_U$ contains ${\sf EA}+\mathrm B\Sigma_1$. See \cite[Section 5.1]{viss:smal19}.

\begin{theorem}
Suppose $U$ is sequential and let $N: U \rhd {\sf S}^1_2$. 
Then, $U$ is Friedman-reflexive iff, for all $\Sigma^0_1$-sentences $S$, there is an $U$ sentence $B$, such that we have
$U^{\sf e} \vdash S^{\mathfrak I} \iff B$. Moreover, $B$ can always be taken to be of the form $S^I$ for some $N$-cut $I$.

$U$ is effectively Friedman-reflexive iff we can find $B$ \textup(or, if you wish, $I$\textup) effectively.
\end{theorem}

\begin{proof}
Suppose $U$ is sequential.
 If $U$ is Friedman-reflexive, then,  the $S^I$ provided by Theorem~\ref{crc}
 gives us the $B$ we are looking for.

Suppose $U^{\sf e} \vdash S^{\mathfrak I} \iff B$. Then, $U+\verz{S^I \mid \text{$I$ is an $N$-cut}} \vdash B$
and, conversely, $U+B \vdash S^I $, for all $N$-cuts $I$. By Theorem~\ref{mocha}, it follows that $U$ is Friedman-reflexive.
Clearly, if we can find the $B$ effectively, then $U$ is effectively Friedman-reflexive.
\end{proof}

So if we view the $S^{\mathfrak I}$ as a second-order or as an infinitary statement, then Friedman-reflexiveness means
a reduction to first-order or finitary statements. 

\subsection{An Example: the Theory \desca}
We provide an example of an effectively Friedman reflexive theory that is not essentially sententially reflexive.
We call the theory of our example \desca\ (Descending Arithmetic). Giving it a name does make it seem like a definite thing. So, it is good to point out
that the theory does depend on two arbitrarily chosen enumerations. 

Let $S_0,S_1,\dots$ enumerate the $\Sigma^0_1$-sentences and let $I_0,I_1,\dots$ be an effective enumeration of ${\sf S}^1_2$-cuts
such that ${\sf S}^1_2 \vdash I_{n+1} \subseteq I_n$ and such that, for each ${\sf S}^1_2$-cut $J$, we can find
a $k$ such that ${\sf S}^1_2\vdash I_k \subseteq J$. Briefly said, $(I_k)_{k\in \omega}$ is effective, descending, and co-initial with all cuts. 

We note that, because, in sequential theories, we have truth-predicates for formulas with $\rho$-complexity below a given number,
we can take $I_n$ to be the intersection of all definable cuts with $\rho$-complexity $\leq n$.

Let \desca\ be ${\sf S}^1_2 + \verz{S^{I_i}_i \to S_i^{J} \mid i\in \omega \text{ and $J$ is a definable cut}}$.
Clearly, \desca\ is effectively Friedman-reflexive.
We note that \desca\ is a sub-theory of $\paco$ and, thus, locally cut-interpretable in ${\sf S}^1_2$.  

Let us say that a theory $V$ is \emph{restrictedly} Friedman-reflexive iff there is an $n$ and a mapping $A \mapsto \ccu A$, where
$\rho(\ccu A) \leq n$, for all $A$. 
It is easy to see that in case $V$ is a \emph{sequential} restrictedly Friedman-reflexive theory, then, for any $N:V \rhd {\sf S}^1_2$, there is
a formula $C(x)$ such that, for all $A$, we have $V \vdash \ccu A \iff C(\gn A)$, where the G\"odel numbers are chosen
w.r.t. $N$. Another immediate insight is that, if $V$ is essentially sententially reflexive, then $V$ is restrictedly Friedman-reflexive. 



\begin{theorem}\label{norep}\label{checkchecksmurf}
The theory \desca\ is not restrictedly Friedman-reflexive and, hence, not essentially sententially reflexive.
\end{theorem}

\begin{proof}
Suppose \desca\ were restrictedly Friedman-reflexive with bound $k_0$.
Let $\rho({\sf S}^1_2) = k_1$.
Suppose $I_p$ is the first logarithmic cut in the sequence. And let $k_2$ be the
maximum of the $\rho$-complexities of the $S_i$ for $i<p$. Finally, let $k_3$ be the complexity
of a standard $\Sigma^0_1$-truth predicate {\sf true}. Let $k$ be the maximum of $k_0$, $k_1$, $k_2$, $k_3$.
 We pick
$n$ so large that $I_n$ is $\Sigma^0_1$-sound for every consistent extension 
of ${\sf S}^1_2$ with complexity $\leq k$. The existence of such a cut is guaranteed by Theorem~\ref{hulpsmurf}.
We may assume that $n>p$.
Let 
\begin{multline*}
A := {\sf S}^1_2+ \verz{\neg S_i \mid \text{$i< p$ and $S_i$ is false}}+ \\
\verz{\neg\, {\sf true}(S_j) \mid \text{$p \leq j < n$ and $S_j$ is false}}
\end{multline*}
and let $B := {\sf S}^1_2 + \oco_A \top$.
We note that $\rho(A) \leq k$. 

We claim that $A +\neg\, C(\gn B)$ is consistent. Suppose it were not. Then, we would have $A \vdash C(\gn B)$. 
It follows that $\desca +A \vdash  \ccu B$, and, hence, $(\desca+A) \rhd B$. Since, $A$ locally cut-interprets $\paco+A$ and, hence, 
$\desca+A$, we find that
$A \rhd B$. \emph{Quod non}, by the usual no-interpretation version of G2.

By the special property of $I_n$, it follows that \[V:= A +\neg\, C(\gn B)+ \verz{\neg S^{I_n}_i \mid \text{$i\geq  n$ and $S_i$ is false}}\] is
consistent. By Theorem~\ref{sicosmurf},  we have ${\sf S}^1_2 \vdash \neg\, {\sf true}(S) \to \neg\, S^{I_p}$, for any $\Sigma^0_1$-sentence $S$. It follows that
$V$ extends \desca. Hence, $V$ is Friedman-reflexive with \ccz\ as selection function and
$V \vdash \neg \ccu B$.
Since $V$ proves every true $\Pi^0_1$-sentence, including $\oco_B\top$, on $I_n$,
we find $V \rhd B$ and, hence $V \vdash  \ccu B$. A contradiction.
\end{proof}

\begin{question}\label{zeursmurf}\label{q5}
\begin{enumerate}[i.]
\item
Is \desca\ reflexive?
If, against expectation, it turns out to be reflexive, 
can we modify the construction to find a non-reflexive, Friedman-reflexive, sequential theory?
\item
 Is there 
a finitely axiomatised $A$ and $K: A\rhd \desca$, such that, for no $D(x)$ in the $A$ language,
we have, for all $B$ in the $A$-language,  $A \vdash D(\gn B) \iff \graydi_{K,A}B$\,?
Here the numerals are the $K$-numerals.
\item
Is there an RE sequential theory that is Friedman-reflexive but not effectively so?
\item
Suppose $U$ is sequential and restrictedly (effectively) Friedman-reflexive. Does it follow that $U$ is essentially
sententially reflexive?
\end{enumerate}

\noindent
We note that our proof of Theorem~\ref{norep} uses special features of \desca. So, the proof does not
 generalise, in an obvious way, to a proof of a positive answer to (iv).
\end{question}

\subsection{Constraints}
In this subsection, we prove two results that constrain the form of consistent, sequential, Friedman-reflexive theories.

\begin{theorem}\label{restraintsmurf}
 Suppose $U$ is consistent, Friedman-reflexive, sequential and RE. Then, any axiomatisation of $U$ must have axioms
 of $\rho$-complexity $>n$, for any $n$.
 \end{theorem}
 
 \begin{proof}
 Suppose that $U$ is consistent, Friedman-reflexive, sequential and RE and that $U$ has a restricted axiomatisation. 
 We fix $N:U \rhd {\sf S}^1_2$.  Let $A $ be any consistent finitely axiomatised theory such that $U \nrhd A$.
 We find that $U \nvdash \ccu A$. So, $U + \neg\,\ccu A$ is consistent.

We have, for all $N$-cuts $I$, that  $(U+\oco_A^I\top)\rhd A$. So, $U+\oco_A^I\top\vdash \ccu A$.
It follows that $U+\neg\, \ccu A\vdash \opr_A^I\bot$. So, by Theorem~\ref{hulpsmurf},
 we find that $\opr_A\bot$ is true and, thus, that $A$ is inconsistent.  \emph{Quod non.}
\end{proof}
 
So, e.g., neither {\sf PRA}, nor $\mathrm I\Sigma_n$, nor ${\sf ACA}_0$ is Friedman-reflexive. 

\begin{remark}
We note that we could, alternatively, have framed the proof of Theorem~\ref{restraintsmurf} as follows.
Theorem~\ref{hulpsmurf} tells us that each consistent finite extension of $U$ tolerates ${\sf S}^1_2$ plus
all true $\Pi_1$-sentences and, hence, is polyglottic. On the other hand, clearly, there is an $A$ such that
$U \nrhd A$ and so $U+\neg\, \ccu A$ is consistent. By Theorem~\ref{kwaaksmurf}, this last theory is not polyglottic.
A constradiction.
\end{remark}

 \begin{example}
 Since, Peano Corto is both reflexive and mutually locally interpretable with ${\sf S}^1_2$, we find that
 $\paco \mutint \mho({\sf S}^1_2)$. Also, $\mho({\sf S}^1_2)$ is restricted and, hence, not
 Friedman-reflexive. Ergo, Friedman-reflexivity is not preserved by mutual interpretability.
 \end{example}

We have the following insight.
 
 \begin{theorem}\label{smartsmurf}
 Suppose $A$ is finitely axiomatised
 and $U$ is Friedman-reflexive. Suppose further that one of $A$, $U$ is sequential. Then, $A$ and $U$ are irreconcilable.
  \end{theorem}
 
 \begin{proof}
 Suppose $A$ is finitely axiomatised and $U$ is Friedman-reflexive.
 Since, both the property demanded for $A$ and the property demanded $U$ are closed under finite extensions,
 it is sufficient to show that $A$ and $U$, if consistent, are not mutually interpretable. Suppose $A$ and $U$ are consistent.
 
 We first consider the case, where $A$ is sequential. Suppose
 $K:U \rhd A$ and $M:A\rhd U$. Consider any finitely axiomatised $B$ such that $B$ is consistent and $A \nrhd B$.
 For example, we could take $B := ({\sf S}^1_2+\oco_A\top)$. 
 
 We have (a) $A + \neg \ccdu MB$ is consistent, since otherwise 
 \[ A \vdash (A+  \ccdu MB) \rhd (U+ \ccu B) \rhd B.\]
 \emph{Quod non}. 
 We have:
 \[ (U+  \ccdu{MK}B) \rhd (A+ \ccdu MB) \rhd (U+ \ccu B) \rhd B.\]
 So,   $U+  \ccdu {MK}B \vdash  \ccu B$. Ergo,
 $A+  \ccdu {MKM} B \vdash  \ccdu MB$, and hence, $A+ \neg\,  \ccdu MB \vdash \neg \ccdu{MKM}B$.
 It follows that (b) $(M\circ K)': (A + \neg  \ccdu MB) \rhd (A + \neg  \ccdu MB)$, where $(M\circ K)'$ is the interpretation
 based on the same translation as $M\circ K$.

Suppose $N:A \rhd {\sf S}^1_2$. Let $I$ be any $N$-cut in $A$. We have
\[(U+\oco_B^{IK}\top) \rhd (A+\oco_B^I \top) \rhd ({\sf S}^1_2 + \oco_B\top) \rhd B.\]
So $U + \oco_B^{IK}\top \vdash  \ccu B$. It follows that $U+ \neg\ccu B \vdash \opr_B^{IK}\bot$, and, hence, we have
$A+ \neg\ccdu MB \vdash \opr_B^{IKM}\bot$. 

We have shown: (c) for all $N$-cuts $I$, we have $A+ \neg\ccdu MB \vdash \opr_B^{IKM}\bot$. Combining (a), (b), and (c), we find, by
Theorem~\ref{zustersmurf2}, that $\opr_B\bot$ is true and, thus, that $B$ is inconsistent. \emph{Quod non}.

We now turn to the case that $U$ is sequential. Suppose $A \mutint U$. Then, we can find a finitely axiomatised sequential $U_0 \subseteq U$,
such that $A \mutint U_0$ and, hence $U_0 \mutint U$, contradicting the first case. 
 \end{proof}
 
 \noindent 
 An alternative argument, for the case that $U$ is sequential, is to note that ${\sf Seq}(A) \mutint U$, where {\sf Seq} is
 the functor that adds sequentiality to $A$.

We note Theorem~\ref{smartsmurf} implies that there is no Friedman-reflexive and sequential $U$
such that ${\sf S}^1_2 \rhd U$. The reason is that any sequential theory interprets ${\sf S}^1_2$.
Of course, ${\sf S}^1_2$ does interpret the theory of the ordering of the natural numbers. However, that
only gives the trivial interpreter logic that proves $\opr\bot$.
So, one might wonder whether there is a more interesting base for ${\sf S}^1_2$.
We will discuss some hopeful signs that there is in Section~\ref{corema}.

\subsection{Interpreter Logic}\label{wijsneuzigesmurf}
In this subsection, we will be concerned with interpreter logics over sequential Friedman-reflexive bases. 
These interpreter logics satisfy L\"ob's Logic. We will show that, if the base
is, in addition, \emph{effectively} Friedman-reflexive, then the proof of Solovay's Theorem
works with minor adaptations.

\subsubsection{Soundness}
Combining Theorems~\ref{whenloeb} and \ref{smartsmurf}, we find:
 
 \begin{theorem}\label{tuinmansmurf}
  Let $\tupel{A,U}$ be an FM-frame. Suppose that one of $A$, $U$ is sequential. 
 Then, $\Lambda^{\sf fr}_{\tupel{A,U}}$ extends {\sf GL}.
 \end{theorem}

Since we have L\"ob's Logic when either $A$ or $U$ is sequential, we also have explicit solutions for modal equations where the
fixed point variable is guarded. E.g., $A \vdash \graydi \top \iff \neg\, \graysq \graydi \top$.
The G\"odel Fixed Point Lemma plays a role in our proof of this fact, but it is still not
a \emph{direct} application. We use the Fixed Point Lemma in a Rosser-style argument in
the proof of Theorem~\ref{zustersmurf}. Can we prove it more directly?
 We do not know whether our provability-like operator can be
represented in $A$ by a formula. E.g., ${\sf S}^1_2$ and the theory of the ordering of the natural numbers
would provide an example. Surprisingly, we can employ the usual argument in case $U$ is \emph{effectively}
Friedman-reflexive. We give the argument below. Even if we, thus, prove a result that is weaker than
what we already know, we think the alternative proof is of independent interest. E.g., it could have 
other generalisations. We first prove a Fixed Point Lemma.

\begin{theorem}\label{fipole}
Suppose $U$ is sequential and effectively Friedman-reflexive. Let $N:U \rhd {\sf S}^1_2$.
We define $\mathfrak I$ w.r.t. $N$.
Suppose $A(x)$ is a boolean combination of $\Sigma^0_1$-formulas with just $x$ free. Then, there is a $B$ in the $U$-language, such that
 $U^{\sf e} \vdash B \iff A^{\mathfrak I}(\gn B)$.
\end{theorem}

\begin{proof}
Suppose $U$ is sequential and effectively Friedman-reflexive. Let $N$ and $\mathfrak I$ be as in the
statement of the Theorem.

By effectivity, we can find a recursive $F$ that sends any $\Sigma^0_1$-sentences $S$ to 
$S^I$, where $S^I$ is equivalent over $U^{\sf e}$ to $S^{\mathfrak I}$.
We can lift this function to Boolean combinations of $\Sigma^0_1$-sentences. Let's say
the result is $G$. 

Suppose $A(x)$ is a boolean combination of $\Sigma^0_1$-formulas with just $x$ free. 
We write $A(G(x))$ for the result of replacing each $\Sigma^0_1$-component $Sx$ of the
Boolean combination by $\exists y,z,u\, (G_0xyz \wedge S_0yu)$, where $Gxyz$ is a $\Delta^0_1$-formula
such that  $\exists z\,Gxyz$ represents the graph of $G$ and $S_0yu$ is a $\Delta_0$-formula such that
$Sx$ is (equivalent to) $\exists u \, S_0xu$.

We can find a $C$ such that ${\sf EA} \vdash C \iff A(G(\gn C))$, by the G\"odel Fixed Point Lemma.
We note that the Fixed Point Lemma yields a sentence of the form $A(G(t))$, where $t$ is a substitution term.
Since, this term is not really in the language, we have to eliminate it. We do this in the same we as we did for
the function $G$, so that $C$ is again a boolean combination of $\Sigma^0_1$-sentences. 
Let $B := G(C)$. Then, \qedright
\begin{eqnarray*}
U^{\sf e} & \vdash & (C\iff  A(G(\gn C)))^{\mathfrak I} \\
& \vdash & C^{\mathfrak I} \iff (A(G(\gn C)))^{\mathfrak I} \\
& \vdash & B \iff (A(\gn B))^\mathfrak I
\end{eqnarray*}
\end{proof}

\begin{theorem}[Alternative Proof for L\"ob's Principle in the effective Case]
Suppose $U$ is sequential and effectively Fried\-man-reflexive. Let $K:A\rhd U$ be an FM-interpretation.
Then, $\Lambda^{\sf fr}_K$ proves L\"ob's Principle.
\end{theorem}

\begin{proof}
Suppose $U$ is sequential and effectively Friedman-reflexive. Let $K:A\rhd U$ be an FM-interpretation.
Consider any $B$ in the $A$-language.
By Theorem~\ref{fipole}, we can find a $D$ such that
$U^{\sf e} \vdash D \iff \apr^{\mathfrak I}_A (D^K \to B)$. Thus,
$U \vdash D \iff \ppd A {(D^K \to B)}$. Setting $E := D^K$, we find
$A \vdash E \iff \graysq_A(E \to B)$. So, we have one variant of the L\"ob 
Fixed Point in $A$. Since we have {\sf K}4, L\"ob's Principle follows.
\end{proof}

\subsubsection{Solovay's Theorem}
We can prove Solovay's Theorem in case the base theory $U$ is both sequential and effectively
Friedman-reflexive. We first formulate the theorem.

Let $\alpha$ range over $0,1,\dots, \infty$. We define $\graysq^{\, 0} \bot := \bot$,  $\graysq^{\,k+1} \bot := \graysq \graysq^{\, k}\bot$, and
$\graysq^{\,\infty} \bot := \top$ and, similarly, for $\opr$. Suppose $K:A \rhd U$ is an FM-interpretation. Let ${\sf d}(K)$ be the smallest $\alpha$ such that
$A \vdash \graysq_{K,A}^{\,\alpha} \bot$. 

\begin{theorem}[{\small Solovay's Theorem for sequential effectively Friedman-reflexive bases}]\label{solosmurf}
Suppose that $K:A \rhd U$ is an FM-interpretation and that $U$ is effectively Friedman-reflexive and sequential. Then,
$\Lambda_K^{\sf fr} = {\sf GL} + \opr^{\,{\sf d}(K)} \bot$.
\end{theorem}

We will prove Solovay's Theorem by verifying the conditions for it given in \cite{dejo:proo91}. 
See also\cite{viss:arit15}.

The idea of  de Jongh, Jumelet and Montagna is that Solovay's embedding result can be verified in an extension of the
 modal logic ${\sf R}^{-}$ enriched with certain fixed points.
We  introduce the logic ${\sf R}^{-}$ of Guaspari and Solovay. See \cite{guas:ross79}.
The language of ${\sf R}^{-}$ is given by:
\begin{itemize}
\item
$\alpha :: = p_0 \mid p_1 \mid \ldots$
\item
$\phi ::= \alpha \mid \bot \mid \top \mid \neg\,\phi \mid \opr \phi \mid (\phi\wedge \phi) \mid 
(\phi \vee \phi) \mid (\phi \to \phi) \mid \opr\phi < \opr \phi \mid \opr \phi \leq \opr \phi$
\end{itemize}

\noindent
The logic ${\sf R}^{-}$ is axiomatised by the axioms and rules of {\sf GL} (for the extended language) plus the
following axioms. 

\begin{enumerate}[${\sf R}^{-}$1.]
\item
$\vdash \opr\phi \leq \opr \psi \to \opr \phi$
\item
$\vdash (\opr \phi \leq \opr \psi \wedge\opr \psi \leq \opr \chi) \to \opr \phi \leq \opr \chi$
\item
$\vdash \opr \phi < \opr \psi \iff  (\opr \phi \leq \opr \psi \wedge \neg\, \opr \psi \leq \opr \phi)$
\item
$\vdash \opr\phi \to (\opr\phi \leq \opr\psi \vee  \opr\psi \leq \opr\phi)$
\item
$\vdash \opr \phi \leq \opr \psi \to \opr(\opr \phi \leq \opr \psi )$
\item
$\vdash \opr \phi < \opr \psi \to \opr(\opr \phi < \opr \psi )$
\end{enumerate}

Now consider a finite Kripke model $\mathcal K$ of {\sf GL}
with nodes $0,\dots, n-1$. Here 0 is the bottom node. 
Let $\prec$ be its accessibility relation. We want to `embed' this model in our modal logic.
To realise this purpose, we add constants $\li_i$, for $i<n$, to the language of  ${\sf R}^{-}$ and we extend the schemes
to the extended language. We demand that
the constants satisfy the following equations. We write $j \parallel i$ for $j$ is incompatible with $i$ w.r.t. $\prec$.
\begin{enumerate}[{\sf fp}1.]
\item
$\vdash \; \li_i \iff  (\opr\neg\,\li_i \wedge \bigwedge_{j\succ i} \oco \li_j \wedge \bigwedge_{j \parallel i}\;\; 
\bigvee_{k\preceq i, \; k\parallel j}\opr \neg \,  \li_k < \opr \neg\,  \li_j)$.
\item
For $i\neq j$, we have $\vdash \;  \opr \neg\, \li_i \leq \opr \neg\,\li_j \; \to \; \opr \neg \, \li_i < \opr \neg \, \li_j$. 
\end{enumerate}

We proceed to introduce our intended interpretation of $\leq$ and $<$.
We remind the reader of the witness comparison notation.
We define, for any $C=\exists x\, C_0(x)$\/ and $D=\exists y\, D_0(y)$:
\begin{itemize}
\item
$C\leq D := \exists x\,(C_0(x) \wedge \forall y<x\,\neg D_0(y))$,
\item
$C < D := \exists x\,(C_0(x) \wedge \forall y\leq x\,\neg D_0(y))$,
\end{itemize}

Suppose $K:A \rhd U$ is an FM-interpretation, where $U$ is sequential. As usual we work with 
a fixed $N:U \rhd {\sf S}^1_2$.
\begin{itemize}
\item
We define
 $\ppd AB < \ppd AC$ by $(\apr_AB < \apr_AC)^{I}$, where
$I$ is a cut such that
\[U^{\sf e} \vdash  (\apr_AB < \apr_AC)^{I} \iff (\apr_AB < \apr_AC)^{\mathfrak I}.\]
\item
We define $\graysq_{K,A} B < \graysq_{K,A}C$ by $(\ppd AB < \ppd AC)^{K}$.
\item
We define  $\ppd AB \leq \ppd AC$ and  $\graysq_{K,A} B \leq \graysq_{K,A}C$ similarly.
\end{itemize}

\noindent
We extend the notion of translation of the modal language to the richer vocabulary in the obvious way.

\begin{theorem}
Suppose $K:A \rhd U$ is an FM-interpretation, where $U$ is sequential. Let $N:U \rhd {\sf S}^1_2$.
Then, we have ${\sf R}^-$ for all $K,A$-translations of the modal language.
\end{theorem}

\begin{proof}
Ad ${\sf R}^-$1. We have ${\sf EA} \vdash \apr_A B \leq \apr_A C \to \apr_A B$. It follows that
\[\desca^{\sf e} \vdash  ( \apr_A B \leq \apr_A C \to \apr_A B)^{\mathfrak I}.\] Hence,
$U \vdash   \ppd AB \leq \ppd AC \to \ppd AB$. We may conclude that \[A \vdash \graysq_{K,A} B \leq \graysq_{K,A}C \to \graysq_{K,A}B.\]

The proofs of ${\sf R}^-$2,  ${\sf R}^-$3, and  ${\sf R}^-$4 are similar.

We treat  ${\sf R}^-$5. Let $I$ be a cut such that \[U^{\sf e} \vdash  (\apr_AB < \apr_AC)^{I} \iff (\apr_AB < \apr_AC)^{\mathfrak I}.\]
We have:
${\sf EA} \vdash \apr_AB \leq \apr_AC \to \apr_A( \apr_AB \leq \apr_AC)^{IK}$. So,
\[U^{\sf e} \vdash  (\apr_AB \leq \apr_AC \to \apr_A( \apr_AB \leq \apr_AC)^{IK})^{\mathfrak I}.\]
And, thus,
\[U \vdash  \ppd AB \leq \ppd AC \to \ppd A{( \graysq_{K,A}B \leq \graysq_{K,A}C)}.\]
We may conclude that $A \vdash  \graysq_{K,A}B \leq \graysq_{K,A}C \to \graysq_{K,A}{( \graysq_{K,A}B \leq \graysq_{K,A}C)}$.

The proof of ${\sf R}^-$6 is similar to that of ${\sf R}^-$5.
\end{proof}

To provide the $\ell_i$ we need an extension of Theorem~\ref{fipole}.

\begin{theorem}\label{mufipole}
Suppose $U$ is sequential and effectively Friedman-reflexive. Let $N:U \rhd {\sf S}^1_2$.
We define $\mathfrak I$ w.r.t. $N$.
Suppose we have formulas $A_i(x_0,\dots, x_{k-1})$, for $i<k$, where each $A_i$ is a boolean combination of $\Sigma^0_1$-formulas with just 
$x_0,\dots,x_{k-1}$ free. Then, for  $i<k$, there are $B_i$ in the $U$-language, such that
 \[U^{\sf e} \vdash B_i \iff A_i^{\mathfrak I}(\gn {B_0},\dots,\gn{B_{k-1}}),\] for each $i<k$.
\end{theorem}

The proof of the Theorem is the similar to the proof of Theorem~\ref{fipole} using the Multiple Fixed Point Lemma
as it is available in {\sf EA}.

In the proof of Theorem~\ref{exilim}, we need the assumptions that we work with a single conclusion system and that
the fixed point construction is a usual construction.

\begin{theorem}\label{exilim}
Suppose $K:A \rhd U$ is an FM-interpretation, where $U$ is sequential. Suppose $U$ is effectively Friedman-reflexive. Let $N:U \rhd {\sf S}^1_2$.
Let $\mathcal K$ be a Kripke model for {\sf GL} with nodes $0,\dots,n-1$ and accessibility relation $\prec$, where
$0$ is the bottom node.
Then, we can find sentences $L_i$, for $i<n$, such that:
\begin{enumerate}[a.]
\item
$A\vdash \; L_i \iff  (\graysq_{K,A}\neg\,L_i \wedge \bigwedge_{j\succ i} \graydi_{K,A} L_j \wedge \bigwedge_{j \parallel i}\;\; 
\bigvee_{k\preceq i, \; k\parallel j}\graysq_{K,A} \neg \,  L_k < \graysq \neg\,  L_j)$.
\item
For $i\neq j$, we have $A\vdash \;  \graysq_{K,A} \neg\, L_i \leq \graysq_{K,A} \neg\,L_j \; \to \; \graysq_{K,A} \neg \, L_i < \graysq_{K,A} \neg \, L_j$. 
\end{enumerate}
\end{theorem}

\begin{proof}
Using Theorem~\ref{mufipole}, we can find sentences $\lambda_i$ such that: 
\[
U^{\sf e}\vdash \; \lambda_i \iff  (\underline i = \underline i \wedge
\apr_A{\neg\,\lambda^K_i} \wedge \bigwedge_{j\succ i} \aco_A {\lambda^K_j} \wedge \bigwedge_{j \parallel i}\;\; 
\bigvee_{k\preceq i, \; k\parallel j}\apr_A {\neg \,  \lambda^K_k} < \apr_A{\neg\, \lambda^K_j})^{\mathfrak I}.
\]
Hence,
\[
U\vdash \; \lambda_i \iff  (\ppd A{\neg\,\lambda^K_i} \wedge \bigwedge_{j\succ i} \ccd A {\lambda^K_j} \wedge \bigwedge_{j \parallel i}\;\; 
\bigvee_{k\preceq i, \; k\parallel j}\ppd A {\neg \,  \lambda^K_k} < \ppd A{\neg\, \lambda^K_j}).
\]
Taking $L_i := \lambda_i^K$, we find (a).

Claim (b) follows from the fact that, for $i\neq j$:
\[{\sf EA}\vdash \;  \apr_{A} \neg\, L_i \leq \apr_{A} \neg\,L_j \; \to \; \apr_{A} \neg \, L_i < \apr_{A} \neg \, L_j.\] 
We note that this last fact uses that our proof system is single conclusion and that the $L_i$ are pairwise distinct.
This last insight follows from the fact that the $\lambda_i$ are pairwise distinct. This is because
the first conjunct of $\Lambda_i$ has the form $(\underline i = \underline i)^I$, for some $N$-cut $I$.
\end{proof}

Solovay's Theorem now follows from the results of \cite{dejo:proo91}. 
We have the following corollary.

\begin{corollary}
Suppose $K:A \rhd U$ is a faithful FM-interpretation and $A$ is consistent and $U$ is effectively Friedman-reflexive, sequential, and polyglottic. Then,
$\Lambda_K^{\sf fr} = {\sf GL}$.
\end{corollary}

\begin{proof}
We assume the conditions of the corollary. 
Clearly $A +\graydi_{K,A}^0\top$ is consistent.
Suppose $A+ \graydi_{K,A}^n\top$ is consistent.
So, by polyglotticity, $U+ \ccd A {\graydi_{K,A}^n\top}$ is consistent. 
But, then, by faithfulness, $A +\graydi_{K,A}^{n+1}\top$ is consistent. 
It follows that ${\sf d}(K) = \infty$.
\end{proof}

\begin{corollary}
Suppose $\tupel{A,U}$ is an FM-frame, $A$ is consistent and sequential, and $U$ is RE, effectively Friedman-reflexive, sequential, and polyglottic. Then,
$\Lambda_{A,U}^{\sf fr} = {\sf GL}$.
\end{corollary}

\begin{proof}
By a result of Friedman, if there is an interpretation of an RE theory $U$ in a finitely axiomatised $A$, then there is also a faithful one.
See \cite{smor:nons85} or \cite{viss:faith05}.
\end{proof}

\subsection{Relation between Logics}
We consider the functor ${\sf Seq}$ of Remark~\ref{cattwee}. 
Let $\jmath_V$ be the interpretation based on relativisation to $\dom$.

\begin{theorem}\label{freubelsmurf}
Suppose $U$ is sequential and Friedman-reflexive and $K:A \rhd U$.
Then, $\Lambda^{\sf fr}_{\jmath_A \circ K } \subseteq \Lambda^{\sf fr}_K$.
\end{theorem}

\begin{proof}
The theorem is a direct application of Theorem~\ref{machtigesmurf} with
$\jmath_A$ in the role of $\eta_A$.
\end{proof}

\section{Concluding Remarks}\label{corema}

In this section, we look briefly backward at what is and what is not achieved and
we look forward at a possible next step in the program.

\subsection{Coordinate-free?}
Did we succeed in giving a treatment of the Second Incompleteness Theorem and of
provability logic that is indeed coordinate-free? 

The results of our general framework as, for example, provided in Sections~\ref{basics},
\ref{intlog} and \ref{compunde} clearly do not involve arithmetisation, neither in their statement nor
in their proof (with the exception of the proof of Theorem~\ref{polyq}). 
In the remaining sections we have general results that do not depend on
arithmetisation in their statement but that do ask for a proof involving arithmetisation.

Sometimes the concrete statement of an application does involve arithmetisation.
This is not unlike the situation for the Cartesian product in Category Theory.
The general treatment of the product is clearly implementation-free, but if we want
to apply it, e.g., to the hereditarily finite sets, we have to define the category from
the sets and for this we need to code pairing \dots 

A first point of discussion is that the usual finite axiomatisations of many salient finitely axiomatisable theories
employ a truth predicate. Examples are ${\sf S}^1_2$, {\sf EA}, and $\mathrm I\Sigma_1$.\footnote{Similarly,
the finite axiomatisation of predicative comprehension over a pair theory does involve various implementation
details.}
So, statements involving these theories would not be coordinate-free to begin with.
Now, of course, we could also axiomatise these with the first so-and-so-many instances of
their coordinate-free schematic axiomatisations. That would be a bit unnatural perhaps, but it
would be at least a coordinate-free specification. Maybe a better way of looking at the matter
is as follows. All finite axiomatisations of a theory
are provably equivalent. Since, our framework works modulo provable equivalence,
we are not dependent on a specific finite axiomatisation. Moreover, the class of all finite axiomatisations is
uniquely determined by the infinite axiomatisation which in the cases above is coordinate-free.
 
Here are some further examples as food for thought.

\begin{itemize}
\item
 Peano Corto is interpretable in ${\sf S}^1_2+\oco_{{\sf S}^1_2}\top$. However,
  neither ${\sf S}^1_2+\oco_{{\sf S}^1_2}\top$ itself nor the Henkin interpretation that
  we employ is coordinate-free.
  \item
  Peano Corto is interpretable in {\sf EA}, since {\sf EA} interprets ${\sf S}^1_2+\oco_{{\sf S}^1_2}\top$
  on a super-logarithmic cut. Here {\sf EA} is coordinate-free (ignoring to the worry articulated above) 
  but the only interpretation that we know of is not. 
  \item
  By Vaught's Theorem (see \cite{vaug:axio67,viss:vaug12}), there is a finite axiom scheme that axiomatises
  Peano Corto. We replace the schematic variables in the scheme by class variables and take the universal
  closure. Also we add Predicative Comprehension. This results in a finitely axiomatised theory, say $W$,
  that conservatively
  extends Peano Corto in an extended language. One can show that $W$ is mutually interpretable
  with {\sf EA}. The interpretation of {\sf EA} in $W$ can be taken to be coordinate-free: we have {\sf EA} on the intersection of
  all cuts that are classes. The interpretation of Peano Corto in $W$ is also coordinate-free, however the specification of
  $W$ is not, since the Vaught construction uses truth-predicates.
  \end{itemize}
  
We note that, if we only know arithmetisation-involving interpretations, but both theories are coordinate-free, then
the frame properties are unproblematic.
For example, our version of G2 for {\sf EA} and Peano Corto is ${\sf EA} \nrhd (\paco+ \ccu{\sf EA})$, which is
perfectly fine, and, similarly, for the fact that the frame-logic of $\tupel{{\sf EA},\paco}$ is {\sf GL},

\begin{remark}
We consider the following theory {\sf IIA} (Initial Isomorphism Arithmetic).
We start with ${\sf S}^1_2$. We expand the language with second-order variables and we expand the theory
with predicative comprehension obtaining ${\sf PC}({\sf S}^1_2)$. Now we expand the language with a 
new unary predicate {\sf J} and a binary predicate {\sf F}. We add an axiom 
$\forall I\, ({\sf cut}(I) \to {\sf J} \subseteq I)$ and an axiom stating that {\sf F} is an isomorphism
between {\sf ID} and {\sf J}.

By Theorem~5.9 of \cite{viss:pean14}, it follows that {\sf IIA} and Peano Basso have the same arithmetical consequences.
So, the pair {\sf IIA} and Peano Basso form a nice coordinate free pair.\footnote{We could even take ${\sf PA}^-$
as starting point rather than ${\sf S}^1_2$, so removing all doubts regarding coordinate freedom. In fact we need only add closure
under $\omega_1$ to the ${\sf PA}^-$-based version of {\sf IIA} to obtain precisely {\sf IIA}. Moreover, we could
add $\omega_1$ using a new symbol and adding appropriate recursion axioms, thus even avoiding sequence coding and the like.}
 
 Since ${\sf PC}({\sf S}^1_2)$ is mutually interpretable with ${\sf S}^1_2+\oco_{{\sf S}^1_2}\top$ and 
 ${\sf S}^1_2+\oco_{{\sf S}^1_2}\top$ is mutually interpretable with {\sf EA}, we find that
 {\sf IIA} interprets {\sf EA}. 
\end{remark}  
  
  \begin{question}\label{q4}
  \begin{enumerate}[i.]
  \item
  Can we find a more canonical axiom scheme for Peano Corto in a coordinate-free way? We note that the cuts are already schematic.
  The whole problem is in replacing the schematic variable that ranges over \emph{$\Sigma^0_1$-sentences} by an unrestricted
 one that ranges over arbitrary formulas.
  \item
  Is there a coordinate-free specification of an interpretation of Peano Corto in {\sf EA} or in some $\mathrm I\Sigma_n$?
  \item
  Do we have ${\sf EA} \rhd {\sf IIA}$?
  \end{enumerate}
  \end{question}

\subsection{New Insights}
Finding coordinate-free representations is a worthy aim, but it should not stand alone.
We also want new insights. 
The present paper does indeed produce some new insights.

For example, the usual form of G2, for the case of {\sf EA}, is  (a) ${\sf EA} \nrhd ({\sf S}^1_2+\oco_{\sf EA} \top)$.
However, we do have   ${\sf EA} \rhd ({\sf S}^1_2+\aco_{\sf EA} \top)$. Our version of G2 with base
$\paco$ is (b) ${\sf EA} \nrhd (\paco+ \ccd{\sf EA}\top) = (\paco+\aco_{\sf EA} \top)$.
We see that in (a) the base theory is weaker and the consistency statement stronger.
In (b) is is the other way around. Both (a) and (b) follow from a version of G2 due to Pudl\'ak: (c)
${\sf EA} \nrhd \mho ({\sf EA}) := ({\sf S}^1_2 + \verz{ \oco_{n,{\sf EA}}\top \mid n\in \omega})$.
We suspect that the version of (b) with $\paco$ replaced by \desca\ does not directly follow from (c).
However, this depends on a  negative answer to Question~\ref{q5}(i).

The most convincing example of a new phenomenon is Solovay's Theorem for interpreter logics
over an effectively Friedman-reflexive base.

\subsection{Perspectives}
We think the obvious next step in the project should be the study of Friedman-reflexivity for pair theories.
There is hope for progress, since Fedor Pakhomov suggested a very natural construction of an effectively Friedman-reflexive pair theory.
This theory is interpretable in ${\sf S}^1_2$.

Of course, the present paper also left a list of open questions. We collected them in  Appendix~\ref{vraagsmurf}.

\appendix

\section{Notations, Notions and Imported Results}\label{notarissmurf}
Theories will be theories in predicate logic of finite signature. We allow theories with sets of axioms of arbitrary complexity.
The variables $T,U,V,\dots$ will range over theories. The variables $A,B,C,\dots$ will ambiguously
range over sentences and finitely axiomatised theories.

Interpretations will be multi-dimensional piece-wise interpretations. If we assume that a theory proves
that there are at least two objects the piece-wiseness can be eliminated. A property we use in the paper
is that predicate logic can interpret the theory of any finite model.
The reader is referred to our
paper \cite{viss:smal19} for a quick introduction to the details. The idea of a piece-wise interpretation
is explained in \cite{viss:onq17} and in \cite{viss:whyR14}.

We will use depth-of-quantifier-alternations as our measure $\rho$ of complexity of formulas.
The notion is worked out in detail in \cite{viss:smal19}.

Here are our main notions and notations.
\begin{itemize}
\item
$U \subseteq V$ means that the theory $V$ considered as a set of theorems extends the theory $U$ considered
as a set of theorems, where $V$ has the same language as $U$.
\item
We write $\opr_\alpha B$ for the formalised statement that $B$ is provable from the theory with axiom-set
represented by $\alpha$. In case $A$ is a finitely axiomatised theory, we write $\opr_A B$ for
$\opr_{\alpha_0} B$, where $\alpha_0 = \bigvee_{i<n} x= \gn{A_i}$, where $A_0,\dots, A_{n-1}$ are the
axioms of $A$. We write $\oco$ for $\neg\opr\neg$.
\item
We write $\opr_{n,A} B$ for the formalised statement that $B$ is provable from $A$ using only statements
of $\rho$-complexity $\leq n$, where $\rho$ measures the depth of quantifier alternations. 
We will usually implicitly assume that $n \geq \rho(A)$.
\item
We write $\apr_AB$ for $\opr_{\rho(A \to B),A}B$.
\item
We write $K:U \rhd V$ or $K:V \lhd U$ for $K$ is an interpretation of $V$ in $U$. We write $U \rhd V$ for $\exists K\, K:U\rhd V$, etcetera,
We write $U \mutint V$ for $U$ and $V$ are mutually interpretable. In other words,  $U \mutint V$ iff $U \rhd V$ and 
$U \lhd V$.\footnote{We used to employ $\equiv$ for mutual interpretability. However, in the
writings of Lev Beklemishev and his school, $\equiv$ is used for extensional sameness of theories.}
In the context of a category of interpretations, we will use $V \stackarrow{K}U$, for $K:V \lhd U$.
\item
 $\mathbb E :={\sf INT}^+_3$ is the category of theories and interpretations, where  
 two interpretations $K,K': V \to U$ are the same iff, for all
$V$-sentences $A$, we have $U \vdash A^K \iff A^{K'}$. 
\item
${\sf E}_{VU}$ is the interpretation based on the identical translation that witnesses $V \subseteq U$.
\item
$U$ and $V$ are \emph{sententially congruent} or \emph{elementary congruent} iff they are isomorphic in $\mathbb E$.
\item
A theory $U$ is \emph{restricted} if, for some $m$ all its axioms are of $\rho$-complexity $\leq m$.
\item
Let an interpretation $N:{\sf S}^1_2 \lhd U$ be given.
A definable $N$-cut in a number theory $U$ is given by a formula $I$ such that $U$ proves that $I$ is a subclass of $N$ and is closed under
$0_N$, ${\sf S}_N$, $+_N$, $\times$ and $\omega_{1N}$ and downwards closed w.r.t. $\leq_N$. The formula defining $I$ need not be
of the form $J^N$. If $I$ is definable inside $N$, the cut is $N$-internal. In case $N$ is a multi-dimensional interpretation, the cut
$I$ is also multi-dimensional. Similarly, if $N$ is piecewise, then $I$ need not be given by a single formula but by a number of pieces.
Since, we need our notion mostly in the context of sequential theories these fine points can be safely ignored.  
\end{itemize}

We turn to the statement of some central external results that we employ in the paper.
Let {\sf true} be a standard $\Sigma^0_1$-truth predicate.
The following theorem is a direct consequence of the estimate of the transformation of witnesses,
when we move from $S$ to ${\sf true}(S)$, for $\Sigma^0_1$-sentences $S$.
See \cite[V.5(b)]{haje:meta91} for details. We can give a similar estimate for the case
of provability.
\begin{theorem}\label{sicosmurf}
Let  $J$ be a logarithmic cut in ${\sf S}^1_2$. Then,
\begin{enumerate}[a.]
\item
 ${\sf S}^1_2 \vdash S^J \to {\sf true}(\gn S)$.
\item
${\sf S}^1_2 \vdash S^J \to \opr_{m,{\sf S}^1_2} S$, for sufficiently large $m$.\\
\sbra{The number $m$ will be ${\sf max}(\rho({\sf S}^1_2), \rho(S))$ plus some constant for the overhead.}
\end{enumerate}
\end{theorem}

The following theorem is a watered-down and inessentially modified version of Theorem~6.2. of \cite{viss:smal19}.
This theorem  is an extension of earlier results independently obtained by Harvey Friedman (see \cite{smor:nons85})
and Jan \kraj\ (see \cite{kraj:note87}). See also \cite{viss:unpr93,viss:faith05}.

\begin{theorem}\label{hulpsmurf}
Suppose $U$ is a consistent restricted sequential RE theory with bound $m$. Let $K$ be a witness of sequentiality for
$U$. Let $N:{\sf S}^1_2 \lhd U$.
 Then, from the data $m$, $K$, and $N$, we can effectively find an $N$-cut $I$ such that, for all $\Sigma^0_1$-sentences $S$, 
 if $U \vdash S^I$, then $S$ is true.
\end{theorem}

The following theorem is Theorem~4.15 of \cite{viss:faith05}

\begin{theorem}\label{zustersmurf}
Suppose $A$ is a consistent, finitely axiomatised sequential theory. Let $U$ be RE. Suppose $A \mutint U$. Then, there is
an $N: {\sf S}^1_2 \lhd U$, such that $N$ is $\Sigma^0_1$-sound. In other words, $U$ tolerates ${\sf S}^1_2$ plus all true
$\Pi^0_1$-sentences. 
\end{theorem}

In our paper, we use the following immediate consequence of Theorem~\ref{zustersmurf}.

\begin{theorem}\label{zustersmurf2}
Suppose $A$ is consistent, finitely axiomatised and sequential. 
Suppose further that $K: A \rhd A$ and $N:A \rhd {\sf S}^1_2$. Then there is
an $N$-cut $I$ such that $K\circ I$ is $\Sigma^0_1$-sound.
\end{theorem}

\begin{proof}
Let $U$ be axiomatised by the $B$ such that $A \vdash B^K$. Clearly, $A \mutint U$. Let $N'$ be the
interpretation promised by Theorem~\ref{zustersmurf}. Take $I$ the common cut of $N$ and $N'$. 
\end{proof}

\begin{remark}
It is easy to see that Theorem~\ref{zustersmurf2} immediately implies Theorem~\ref{zustersmurf}. In fact, 
the natural order is to prove Theorem~\ref{zustersmurf2} first.
\end{remark}

Finally, we sketch some of the details of the treatment of the functor {\sf Seq}.
We remind the reader that the mapping {\sf Seq} is defined as follows.
We start with theory $V$. We extend the signature of $V$ with a unary predicate $\dom$ and
a binary predicate $\in$. The axioms of ${\sf Seq}(V)$ are the axioms of $V$ relativised to
$\dom$ plus the unrelativised axioms of Adjunctive Set Theory {\sf AS} plus the axiom that says that
all elements of $\dom$ are empty sets. $\eta_V$ is the interpretation of $V$ in ${\sf Seq}(V)$ given by
relativisation to $\dom$.

\begin{theorem}
Suppose $V \stackarrow{K} W$. Then there is an interpretation ${\sf Seq}(K)$ such that ${\sf Seq}(V) \stackarrow{{\sf Seq}(K)} {\sf Seq}(W)$.
\end{theorem}

\begin{proof}
Suppose $V \stackarrow{K} W$.
Consider the interpretation $K' := \theta_V \circ K$ of $V$ in ${\sf Seq}(W)$. We note that $K$ may be piece-wise and multidimensional
(with possibly different dimensions for the pieces). 
Since in ${\sf Seq}(W)$ we have the full machinery of sequences available, we can rebuild $K'$ to a one-dimensional, one-piece interpretation
$K^\ast$ that is ${\sf Seq}(W)$-provably definably isomorphic to $K'$. We extend $K^\ast$ to the interpretation $M :={\sf Seq}(K)$ as
follows.
\begin{itemize}
\item
$M$ has two pieces $\mathfrak s$ and $\mathfrak o$ both of dimension 1.
\item
$\delta_{M}^{\mathfrak o}$ is $\delta_{K^\ast}$ and $\delta_{M}^{\mathfrak s}$ is the complement of $\dom$.
\item
Identity between elements of  different pieces is always $\bot$. We have: $x =_{M}^{\mathfrak{oo}} y$ iff 
$x$ and $y$ are in $\delta_{M}^{\mathfrak o}$ and $x =_{K^\ast} y$, and   $x =_{M}^{\mathfrak{ss}} y$
 iff 
$x$ and $y$ are in $\delta_{M}^{\mathfrak s}$ and $x = y$.
\item
$\dom_M^{\mathfrak o}$ is $\delta_{M}^{\mathfrak o}$ and $\dom_M^{\mathfrak s}$ is empty.
\item
Let $P$ be an $n$-ary predicate of the $V$-language. Then $P_{M}^{\mathfrak o,\dots,\mathfrak o}(\vec x\,)$ tells us that
each $x_i$ is in $\delta_M^{\mathfrak o}$ and that $P_{K^\ast}(\vec x\,)$. If any variable in $P(\vec x\,)$ is assigned
a non-$\mathfrak o$ piece, then $P_M(\vec x\,)$ is false.
\item
We have:
\begin{itemize}
\item
$x \in_{M}^{\mathfrak{oo}} y$ and $x \in_{M}^{\mathfrak{so}} y$ are false. 
\item
$x \in_{M}^{\mathfrak{os}} y$ iff $x\in \delta_M^{\mathfrak o}$, $y\in \delta_M^{\mathfrak s}$, and
there is an $x'$ with $x=_{K^\ast} x'$ and ``$\tupel{0,x'} \in y$''. Here the scare quotes are there to
remind us that we do not have extensionality. So we really mean here: there is an empty set $z$ and there
is a pair of the form $\tupel{z,x'}$ such that \dots
\item
$x \in_{M}^{\mathfrak{ss}} y$ iff $x\in \delta_M^{\mathfrak s}$, $y\in \delta_M^{\mathfrak s}$, and
 ``$\tupel{1,x} \in y$''.
\end{itemize}
\end{itemize}
 It is easy to verify that the interpretation sketched here indeed is an interpretation of ${\sf Seq}(V)$ in $W$. 
\end{proof}

It would be nice if we could prove that {\sf Seq} lifts to a functor in $\mathbb E$. However,
we do not quite see that how that can work for the $\in$-part.

\begin{question}\label{q10}
Does {\sf Seq} or, some appropriate variant of it, lift to a functor on $\mathbb E$? 
\end{question}

\begin{theorem}
Suppose $V$ is sequential. Then, ${\sf Seq}(V) \mutint V$.
\end{theorem}

\begin{proof}
We have $V \stackarrow{\eta_V} {\sf Seq}(V)$. 
 Let $\in^\star$ be a $V$-formula that witnesses that $V$ is sequential.
 Our interpretation $M$ of ${\sf Seq}(V)$ in $V$ looks as follows.
 We duplicate the domain by having two pieces $\mathfrak s$ and $\mathfrak o$ with
 domain $x=x$. We let $\dom$ correspond to the elements in the $\mathfrak o$-piece
 and have the predicates of $V$ on the $\mathfrak o$-piece. We take $x\in_M^{\mathfrak {os}} y$ iff $x$ and $y$ are
 in the correct domains and
 ``$\tupel{0,x} \in^\star y$''. Similarly, $x\in_M^{\mathfrak {ss}} y$ iff $x$ and $y$ are
 in the correct domains and
 ``$\tupel{1,x} \in^\star y$''. We set $\in_M$ to $\bot$ in the other cases.
\end{proof}

\section{Proof of Theorems~\ref{gelijkheidssmurf} and \ref{machtigesmurf}}\label{verismurf}
This appendix contains the proofs of Theorems~\ref{gelijkheidssmurf} and \ref{machtigesmurf} plus some
discussion on an alternative more categorical formulation of Theorem~\ref{machtigesmurf}.

\subsection{Proof of Theorem~\ref{gelijkheidssmurf}}\label{progel}
 Suppose $U_0$ is Friedman-reflexive and $A_0$, $A_1$ are finitely axiomatised.
 Suppose further that $U_0\stackarrow{K_0}A_0$, $U_1\stackarrow{K_1}A_1$ and
 $K_0 \approx K_1$. We want to show that $U_1$ is Friedman-reflexive and
 that  $\Lambda^{\sf fr}_{K_0} = \Lambda^{\sf fr}_{K_1}$. 
 
Since,  $K_0 \approx K_1$, we can find $M,\breve M, P,\breve P$ such that the following
diagram commutes.
  \[
 \begin{tikzcd}
 U_0 \rar{K_0}\arrow[leftrightarrow]{d}[swap]{M, \breve M} & A_0 \arrow[leftrightarrow]{d}{P,\breve P} \\
 U_1 \rar{K_1} & A_1
 \end{tikzcd}
 \]
 Here $M$ and $P$ correspond to the down-direction of the isomorphism and $\breve M$ and
 $\breve P$ go up.
 
 The fact that $U_1$ is Friedman reflexive is immediate from Theorem~\ref{terugtreksmurf}.
  
 \begin{lemma}\label{atocom}
 $A_0 \vdash \graydi_{K_0,A_0} B \iff \graydi_{K_1,A_1}^{\breve P} B^P$ and  $A_1 \vdash \graydi_{K_1,A_1} D \iff \graydi^{P}_{K_0,A_0} D^{\breve P}$.
 \end{lemma}
 
 \begin{proof}
 By symmetry, we only have to prove the first conjunct of the lemma.  We first prove:
 \[ (\dag)\;\;\; U_0 \vdash \ccd {(U_0),A_0} B \iff \cct{\breve M} {(U_1),A_1}{B^P}.\]
 We have:
 \begin{eqnarray*}
 U_0 + \cct{\breve M} {(U_1),A_1}{B^P} & \rhd & U_1+ \ccd{(U_1),A_1} {B^P} \\
 & \rhd &  (A_1 + B^P)  \\
 & \rhd & (A_0+B)
 \end{eqnarray*}
 It follows that
 \[ (\ddag)\;\;\; U_0 +  \cct{\breve M} {(U_1),A_1}{B^P} \vdash \ccd{(U_0),A_0}  B.\]
 By symmetry, we find 
  $U_1 +  \cct{M} {(U_0),A_0}{D^Q} \vdash \ccd{(U_1),A_1} D$.
  Substituting $B^P$ for $D$ gives  $U_1 +  \cct{M} {U_0,A_0}{ B^{PQ}} \vdash \ccd{(U_1),A_1}{B^P}$.
  By Theorem~\ref{minismurf}, we may replace provable equivalents under $\ccz_{(U_0)}$, so:
  $U_1 +  \cct{M}{(U_0),A_0}{B} \vdash \ccd{(U_1),A_1}{B^P}$.
  It follows that \[U_0 +  \cct{M\breve M}{(U_0),A_0}{B} \vdash \cct{\breve M}{(U_1).A_1}{B^P}\]
   and, hence,
  \[\$ \;\;\;\;\; U_0 +  \ccd{(U_0),A_0}{B} \vdash \cct{\breve M}{(U_1),A_1}{B^P}.\]
  Combining (\ddag) and (\$), we find (\dag). In its turn (\dag) gives us:\qedright
  \begin{eqnarray*}
  A_0 \vdash \graydi_{K_0,A_0} B & \iff & \cct{K_0}{(U_0),A_0}{B} \\
  & \iff & \cct{\breve MK_0}{(U_1),A_1}{B^P} \\
   & \iff & \cct{ K_1\breve P}{(U_1),A_1}{B^P} \\
   & \iff & \graydi_{K_1,A_1}^{\breve P} B^P
   \end{eqnarray*}
\end{proof}

Suppose $\tau_i$ is a function from the propositional atoms to the $A_i$-language.
To simplify notations a bit we will confuse, e.g., $P$ and the mapping $D \mapsto D^P$.

\begin{lemma}
We have: 
\[A_0 \vdash \phi^{(\tau_0,K_0)} \iff \phi^{(P\circ\tau_0,K_1)\breve P}\text{ and }
A_1 \vdash \phi^{(\tau_1,K_1)} \iff \phi^{(\breve P\circ\tau_1,K_0) P}.\]
\end{lemma}

\begin{proof}
We prove the first conjunct.
Our proof is by induction on $\phi$. In the atomic case we have:
\begin{eqnarray*}
A_0 \vdash p^{(\tau_0,K_0)} &\iff & \tau_0(p) \\
& \iff & (\tau_0(p))^{P\breve P}\\
& \iff &  p^{(P\circ\tau_0,K_1)\breve P}
\end{eqnarray*}
The cases of the truth-functional connectives are simple. Finally, for the case of $\oco$, we employ Lemma~\ref{atocom}.
\qedright
\begin{eqnarray*}
A_0 \vdash (\oco\psi)^{(\tau_0,K_0)} & \iff & \graydi_{K_0,A_0}\psi^{(\tau_0,K_0)}\\
& \iff & \graydi_{K_0,A_0} \psi^{(P\circ\tau_0,K_1)\breve P}\\
& \iff & (\graydi_{K_1,A_1} \psi^{(P\circ\tau_0,K_1)\breve P P})^{\breve P}\\
& \iff & (\graydi_{K_1,A_1}\psi^{(P\circ\tau_0,K_1)})^{\breve P} \\
& \iff & (\oco\psi)^{(P\circ\tau_0,K_1)\breve P}
\end{eqnarray*}
\end{proof}

Finally, we prove the theorem.
We have:
\begin{eqnarray*}
 A_1 \vdash \phi^{(P\circ\tau,K_1)} 
& \To & A_0 \vdash  \phi^{(P\circ\tau,K_1)\breve P} \\
& \To &  A_0 \vdash \phi^{(\tau,K_0)}
\end{eqnarray*}

It follows that, if $\Lambda_{K_1}^{\sf fr} \vdash \phi$, then $\Lambda_{K_0}^{\sf fr} \vdash \phi$. By symmetry, we also have the other direction.

\subsection{Proof of Theorem~\ref{machtigesmurf}}
Let $F$ be an endofunctor of $\mathbb D$. Here we assume that $F$ operates on concrete
theories and interpretations. Suppose that:
\begin{enumerate}[A.]
\item
$F$ preserves finite axiomatisability.
\item
 For each theory $V$, we have a faithful interpretation  $V \stackarrow{\eta_V} F(V)$.
 \item
 Suppose $\Gamma$ is a set of sentences in the $V$-language.
 Then, $F(V+\Gamma) = F(V) + \Gamma^{\eta_V}$.
\end{enumerate}
We note that we can drop the `faithful' in Condition (B), if we demand that $F$ preserves consistency.

Let $U$ be Friedman-reflexive and suppose $F(V) \stackarrow{M_V} V$, for all extensions
$V$ of $U$ in the same language.
Let $A$ be finitely axiomatised and suppose  $U \stackarrow{K} A$.

We show that $\Lambda^{\sf fr}_{\eta_A \circ K} \subseteq \Lambda^{\sf fr}_K$.

\begin{proof}
For some $P$, we have $(A+B) \stackarrow{P} (U+\ccd AB)$. Let $M' := M_{U+\ccd AB}$.
We have:
\[ 
\begin{tikzcd}
(F(A)+B^{\eta_A})  =  F(A+B)  \rar{F(P)} & F(U+\ccd AB) \rar{M'} & (U+\ccd AB) 
\end{tikzcd}
\]
So, (a) $U \vdash \ccd AB \to \ccd{F(A)}{B^{\eta_A}}$.

For some $Q$, we have $Q: (F(A)+B^{\eta_A}) \stackarrow{Q} (U+ \ccd{F(A)}{B^{\eta_A}})$.
So, we have: 
\[ 
\begin{tikzcd}
(A+B) \rar{\eta_{A+B}} & F(A+B) = (F(A)+B^{\eta_A}) \rar{Q} & (U+ \ccd{F(A)}{B^{\eta_A}})
\end{tikzcd}
\]
So, (b)
$U \vdash \ccd{F(A)}{B^{\eta_A}} \to \ccd AB$. Combining (a) and (b), we find:
\[U \vdash \ccd{F(A)}{B^{\eta_A}} \iff \ccd AB.\]

\noindent
We may conclude that: 
\[ (\dag)\;\;\; F(A) \vdash \graydi_{\eta_A\circ K,F(A)}B^{\eta_A} \iff \graydi^{\eta_A}_{K,A}B.\]
Let $\sigma$ be any function from propositional variables to sentences of the $A$-language.
To lighten our notational burdens a bit, we will confuse $\eta_A$ and the mapping: $D \mapsto D^{\eta_A}$.
By induction, we prove that $F(A) \vdash \phi^{(\eta_A \circ\sigma, \eta_A\circ K)} \iff (\phi^{(\sigma,K)})^{\eta_A}$.
We treat the case of $\oco$. Let $\phi := \oco\psi$. Using (\dag), we find:
\begin{eqnarray*}
F(A) \vdash \phi^{(\eta_A \circ\sigma, \eta_A \circ K)} & \iff & 
 \graydi_{\eta_A \circ K,F(A)}\psi^{(\eta_A \circ\sigma,\eta_A \circ K)} \\
 & \iff &
 \graydi_{\eta_A \circ K, F(A)}(\psi^{(\sigma,K)})^{\eta_A} \\
 & \iff & \graydi^{\eta_A}_{K,A}\psi^{(\sigma,K)}  \\
 & \iff &  (\phi^{(\sigma,K)})^{\eta_A}
\end{eqnarray*}
From the fact that $\eta_A$ is faithful, we now have:
\begin{eqnarray*}
(\ddag)\;\;  F(U) \vdash \phi^{(\eta_A\circ \sigma, \eta_A\circ K)}  & \Iff  & F(U) \vdash (\phi^{(\sigma,K)})^{\eta_A} \\
& \Iff & U \vdash \phi^{(\sigma,K)}
\end{eqnarray*} 
Our theorem is immediate from (\ddag).
\end{proof}

\begin{remark}
The conditions for Theorem~\ref{machtigesmurf} look somewhat \emph{ad hoc}. Especially,
Condition (C) goes into the hardware in an unelegant way. The following (more specific) conditions look a little bit better.
However, they have he disadvantage that they do
not (yet) apply to our main application: we did not supply a version of  {\sf Seq} that is an endofunctor of $\mathbb E$.

We work in category $\mathbb E$ enriched with designated embedding arrows $V\stackarrow{{\sf E}_{VW}}W$
between theories of the same signature that are based on the identity translation.
We demand that $F$ is a functor on the enriched category that preserves (modulo sameness)
the embedding arrows. Let $\bbot$ be the theory in the language of identity axiomatised by $\bot$.

Our new conditions are as follows. 
\begin{enumerate}[A${}^\star$.]
\item
$F$ preserves finite axiomatisability.
\item
If $\bbot \stackarrow{} F(V)$, then $\bbot \stackarrow{} V$.
 \item
 For each $V$, there is an interpretation  $V \stackarrow{\eta_V} F(V)$ such that:
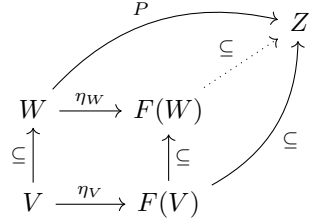
\[
\begin{tikzcd}
&& Z \\
W \rar{\eta_W} \arrow[bend left]{urr}{P} & F(W)\arrow[dotted]{ur}{\subseteq} & \\
V\arrow{u}{\subseteq} \rar{\eta_V} & F(V) \arrow[swap]{u}{\subseteq} \arrow[bend right,swap]{uur}{\subseteq} &
\end{tikzcd}
\]

\noindent
If $V \subseteq W$, then, ${\sf E}_{F(V)F(W)} \circ \eta_V = \eta_W \circ {\sf E}_{VW}$.
Moreover, whenever ${\sf E}_{F(V)Z} \circ \eta_V = P \circ {\sf E}_{VW}$, then
$F(W) \subseteq Z$ and $P = {\sf E}_{F(W)Z} \circ \eta_W$. 
\end{enumerate}

\noindent
It is easy to see that the new conditions imply the old ones. \end{remark}

 \section{Local Logics for a Frame}\label{funo}
 There are other notions that can be considered than just the logic associated with an FM-frame or an FM-interpretation.
 Let $\tupel{A,U}$ be an FM-frame.
  Let $X$ range over finite sets of assignments from the propositional variables to $A$-sentences and let
 $K$ range over interpretations $K:A\rhd U$. We define:
 \begin{itemize}
 \item
 $\Lambda^\star_{A,U} := \verz{ \phi \mid \forall X \, \exists K\, \forall \sigma\in X\,A \vdash \phi^{(\sigma,K)}}$. \\
 \item
 $\Lambda$ \emph{is a local logic for $\tupel{A,U}$} if $\Lambda$ is a subset of  
 $\Lambda^\star_{A,U}$ that contains {\sf K}4 and is closed under
 modus ponens, necessitation, and substitution. 
 \end{itemize}
 
 \noindent
 We have the following obvious facts.
 
 \begin{theorem}
 Consider an FM-frame $\tupel{A,U}$. Then,
  $\Lambda^\star_{A,U}$ is closed under substitution and necessitation. Moreover, if
  $\phi$ is in  $\Lambda^\star_{A,U}$ and $\psi$ is in  $\Lambda^{\sf fr}_{A,U}$, then
  $(\phi \wedge \psi)$ is in $\Lambda^\star_{A,U}$.
 \end{theorem}
 
  \begin{theorem}
   Consider an FM-frame $\tupel{A,U}$. 
A local logic for $\tupel{A,U}$ is a logic. Moreover, every $\phi\in \Lambda^\star_{A,U}$ is contained
in a minimal local logic containing it.
  \end{theorem}
  
  \begin{theorem}
   Consider an FM-frame $\tupel{A,U}$. Let $\Lambda$ be a local logic for $\tupel{A,U}$.
   Then, the logic generated by $\Lambda$ and $\Lambda^{\sf fr}_{A,U}$ is a local logic for
   $\tupel{A,U}$.
  \end{theorem}
  
  We proceed with a characterisation of reflection.
 Consider theories $V$ and $W$. We say that $\mutint$ has the \emph{forward property} for
$V,W$ iff, for all $B$ in the $V$-language, there is a $C$ in the $W$-language, such that $(V+B) \mutint (W+C)$.

\begin{theorem}\label{localsmurf}
Consider an FM-frame $\tupel{A,U}$.
Then, $\mutint$ has the forward property for $A$, $U$ iff 
 {\sf S}4 is a local logic for $\tupel{A,U}$.
\end{theorem}

\begin{proof}
Suppose $\mutint$ has the forward property for $A,U$. Consider any sentences $B_0$, \dots, $B_{n-1}$
in the language of $A$. Let $B^\star_0$, \dots, $B^\star_{k-1}$ enumerate all conjunctions of the form $\bigwedge_{i<n} \pm B_i$,
where $\pm B_i$ is either $B_i$ or $\neg B_i$. 

Consider any $B_j^\star$.
We have, for some $C_j^\star$ that $ (A+B_j^\star) \mutint (U+C_j^\star)$.
It follows that $U \vdash C_j^\star \to \ccd{A}{B_j^\star}$.  Ergo, $(A+B_j^\star) \rhd (U+\ccd{A}{B_j^\star})$.
Let $K_j$ be the witnessing interpretation.
It follows that $A+B_j^\star\vdash \graydi_{K_j, A} B_j^\star$. If $B_i$ is unnegated in $B_j^\star$, then 
$A \vdash \graydi_{K_j, A} B_j^\star \to \graydi_{K_j, A} B_i$.
 If $B_i$ is negated in $B_j^\star$, then 
$A +B_j^\star \vdash   B_i\to \graydi_{K_j, A} B_i$.
So,
$A+B_j^\star\vdash \bigwedge_{i<n} (B_i \to \graydi_{K_j,A} B_i)$

Now let $K := K_0\tupel{B^\star_0}(K_1\tupel{B_1^\star} \dots)$. Because the $B_j^\star$ are mutually exclusive, we find that, for each $j<k$, we have 
$A+B_j^\star\vdash \bigwedge_{i<n} (B_i \to \graydi_{K,A} B_i)$ and, thus 
$A\vdash \bigwedge_{i<n} (B_i \to \graydi_{K,A} B_i)$. From this, it is immediate that
$p \to \oco p$ is in $\Lambda_{A,U}^\star$. Moreover, {\sf S}4 is the logic generated by $p\to \oco p$ over {\sf K}4.

\medskip
In the other direction, suppose {\sf S}4 is a local logic for $\tupel{A,U}$. Consider any $B$ in the $A$-language. For some $K:A \rhd U$, we have
 $A+B \vdash \graydi_{K,A}B$. So, $(A+B) \rhd (U+ \ccd AB)$. Moreover, $(U+ \ccd AB) \rhd (A+B)$.
\end{proof}
 
 \begin{question}\label{q11}
 The second part of the proof of Theorem~\ref{localsmurf} only uses a singleton set $X$. We wonder
 whether that means that our approach may be simplified.
 \end{question}
 
 We note that it follows that the logic generated by {\sf S}4 and $\Lambda^\star_{A,U}$ is a local logic for $\tupel{A,U}$.

\section{List of Questions}\label{vraagsmurf}
Here are all questions asked in the paper.

\medskip
{\footnotesize
\begin{enumerate}[{\sf Q}1.]
\item
 Is there a theory $U$ that is consistent, effectively Friedman-reflexive and not strongly essentially reflexive?
 This is Question~\ref{q1}.
 \item
Can we find a more inspiring example of a theory with logic {\sf S}4 than Example~\ref{losersmurf}?
Is it, perhaps, possible to find an FM-interpretation with interpreter logic precisely {\sf S}4?
 This is Question~\ref{q1A}.
\item
Suppose $K,M:U \lhd A$ are FM-interpretations and  ${\sf Th}(K)={\sf Th}(M)$. 
Do we have  $\Lambda^{\sf fr}_K = \Lambda^{\sf fr}_{M}$,
 or is there a counter-example? This is Question~\ref{q2}.
 \item
 Is there an example of an FM-interpretation $K:A\rhd U$, where $U$ is complete, with an interesting interpreter logic?
 This is Question~\ref{q2A}.
 \item
 Is there a consistent finitely axiomatised theory that is effectively Friedman-reflexive?
 This is Question~\ref{q3}.
\item
Is \desca\ reflexive? If, against expectation, it turns out to be reflexive, 
can we modify the construction to find a non-reflexive, Friedman-reflexive, sequential theory?
 This is Question~\ref{q5}(i).
\item
 Is there 
a finitely axiomatised $A$ and $K: A\rhd \desca$, such that, for no $D(x)$ in the $A$ language,
we have, for all $B$ in the $A$-language,  $A \vdash D(\gn B) \iff \graydi_{K,A}B$\,?
Here the numerals are the $K$-numerals.
 This is Question~\ref{q5}(ii).
\item
Is there an RE sequential theory that is Friedman-reflexive but not effectively so?
 This is Question~\ref{q5}(iii).
\item
Suppose $U$ is sequential and restrictedly (effectively) Friedman-reflexive. Does it follow that $U$ is essentially
sententially reflexive?  This is Question~\ref{q5}(iv).
  \item
  Can we find a more canonical axiom scheme for Peano Corto in a coordinate-free way? We note that the cuts are already schematic.
  The whole problem is in replacing the schematic variable that ranges over \emph{$\Sigma^0_1$-sentences} by an unrestricted
 one that ranges over arbitrary formulas. This is Question~\ref{q4}(i).
  \item
  Is there a coordinate-free specification of an interpretation of Peano Corto in {\sf EA} or in some $\mathrm I\Sigma_n$?
   This is Question~\ref{q4}(ii).
  \item
  Do we have ${\sf EA} \rhd {\sf IIA}$?
  This is Question~\ref{q4}(iii).
  \item
Does {\sf Seq} or, some appropriate variant of it, lift to a functor on $\mathbb E$? 
This is Question~\ref{q10}.
\item
 The second part of the proof of Theorem~\ref{localsmurf} only uses a singleton set $X$. We wonder
 whether that means that our approach may be simplified.
This is Question~\ref{q11}.
\end{enumerate}
}
\end{document}